\documentclass{EMClassColection} 

\usepackage{todonotes}
\usepackage[english]{babel} 
\usepackage[utf8]{inputenc}
\usepackage[T1]{fontenc}
\usepackage{verbatim} 
\usepackage{float}

\usepackage{amsmath}
\usepackage{amsfonts}
\usepackage{amssymb}
\usepackage{amsthm}

\usepackage{mathrsfs}
\usepackage{mathtools}
\usepackage{euscript}   
\usepackage[scr=boondoxo,scrscaled=1.05]{mathalfa} 

\usepackage{pifont} 
\usepackage{tikz-cd} 

\usepackage{listings}

\usepackage[colorlinks=true, pagebackref=true]{hyperref}
\hypersetup{linkcolor=black, citecolor=black, urlcolor=black}
\usepackage{comment} 

\newtheorem{theorem}{Theorem}[section]
\newtheorem{proposition}[theorem]{Proposition}
\newtheorem{lemma}[theorem]{Lemma}
\newtheorem{corollary}[theorem]{Corollary}

\newtheorem{obs}[theorem]{Remark}

\usepackage[all]{xy}

\usepackage{bbm}  

\newcommand\EE{\mathbb{E}}

\newcommand\PP{\mathbb{P}}

\newcommand\RR{\mathbb{R}}

\newcommand\ZZ{\mathbb{Z}}
\newcommand\one{\mathbbm{1}} 

\newcommand{\euB}{\EuScript{B}}
\newcommand{\euF}{\EuScript{F}}
\newcommand{\euG}{\EuScript{G}}
\newcommand{\euL}{\EuScript{L}}
\newcommand{\euM}{\EuScript{M}}

\newcommand{\C}{\mathcal{C}}
\newcommand{\D}{\mathcal{D}}

\newcommand{\T}{\mathcal{T}}
\newcommand{\U}{\mathcal{U}}

\newcommand{\cS}{\mathcal{S}}

\newcommand{\tg}{{g}}

\newcommand{\tr}{{\tilde r}}
\newcommand{\terr}{{\tilde r}}
\newcommand{\tte}{{\tilde t}}
\newcommand{\tv}{{\tilde v}}
\newcommand{\tw}{{\tilde w}}
\newcommand{\tx}{{\tilde x}}
\newcommand{\ty}{{\tilde y}}
\newcommand{\tz}{{\tilde z}}

\newcommand{\tvarphi}{{\tilde \varphi}}

\newcommand{\qo}{q}
\newcommand{\po}{p}
\newcommand{\vo}{w}

\newcommand{\sfH}{\mathsf{H}}

\newcommand{\sfX}{\mathsf{X}}
\newcommand{\sfY}{\mathsf{Y}}

\newcommand{\bsfX}{\textsf{\textbf{X}}}

\newcommand{\ttA}{\mathsf{A}}

\newcommand{\ttK}{\mathsf{K}}

\newcommand{\ttu}{\mathsf{u}}
\newcommand{\ttj}{\mathsf{j}}
\newcommand{\tto}{\mathsf{o}}
\newcommand{\tty}{\mathsf{y}}
\newcommand{\ttx}{\mathsf{x}}

\newcommand{\dd}{\mathrm d}

\newcommand{\ba}{\mathsf a}
\newcommand{\bb}{\mathsf b}

\newcommand{\bo}{\mathsf o}
 \newcommand{\boa}{\overline{\mathsf o\mathsf a}}
 \newcommand{\bob}{\overline{\mathsf o\mathsf b}}
  \newcommand{\bab}{\overline{\mathsf a\mathsf b}}

\newcommand{\kX}{\mathfrak{X}}
\newcommand{\kXa}{\mathfrak{X}_{\ge0}}
\newcommand{\kY}{\mathfrak{Y}}
\newcommand{\kYa}{\mathfrak{Y}_{\ge0}}

\newcommand\eps{\varepsilon}
\newcommand{\err}{r}
 \newcommand{\hy}{\hat y}

\newcommand\norm[1]{||{#1}|| }
\newcommand\normi[1]{||{#1}||_\infty}

\newcommand{\eqlaw}{\stackrel{\law}{=}}

\DeclareMathOperator{\eff}{eff}
\DeclareMathOperator{\law}{law}
\DeclareMathOperator\Var{\mathbb{V}}
\DeclareMathOperator\Cov{\text{Cov}}
\DeclareMathOperator{\sign}{sign}

\allowdisplaybreaks

\title{Hard Rod Hydrodynamics \\and the Levy Chentsov Field}
\author{Pablo A. Ferrari, Chiara Franceschini, \\[1mm]Dante G. E. Grevino, Herbert Spohn}
\abstract{We study the hydrodynamics of the hard rod model proposed by Boldrighini, Dobrushin and Soukhov by describing the displacement of each quasiparticle with respect to the corresponding ideal gas particle as a height difference in a related field. Starting with a family of nonhomogeneous Poisson processes contained in the position-velocity-length space $\RR^3$, we show laws of large numbers for the quasiparticle positions and the length fields, and the joint convergence of the quasiparticle fluctuations to a Levy Chentsov field. We allow variable rod lengths, including negative lengths.}
\Keywords{hard rods, hydrodynamics, Levy Chentsov field}
\MSC{2020}{82C21, 
82D15, 
37B15, 
70F45, 
76A05
}
\shortauthor{PA Ferrari, C Franceschini, DGE Grevino, H Spohn} 

\Year{20XX} 
\volume{XX} 
\firstpage{55} 
\doi{XXX} 

\begin{document}

\maketitle

\begin{flushright}
  {\em Dedicated to Errico Presutti\\ at the occasion of his 80th birthday\\ with deep gratitude for his continuing\\
guidance and lasting insights}
\end{flushright}

\paragraph{Abstract} We study the hydrodynamics of the hard rod model proposed by Boldrighini, Dobrushin and Soukhov by describing the displacement of each quasiparticle with respect to the corresponding ideal gas particle as a height difference in a related field. Starting with a family of nonhomogeneous Poisson processes contained in the position-velocity-length space $\RR^3$, we show laws of large numbers for the quasiparticle positions and the length fields, and the joint convergence of the quasiparticle fluctuations to a Levy Chentsov field. We allow variable rod lengths, including negative lengths.\\

\noindent{\sl Keywords:} hard rods, hydrodynamics, Levy Chentsov field\\

\noindent{\sl MSC 2020:}\ 
82C21, 
82D15, 
37B15, 
70F45, 
76A05
\section{Introduction}
 Hard rods is a model of classical particles  moving on the real line and interacting through a hard core potential of diameter  $\mathsf{r}$.
 One is interested in the dynamics of the counting function $f(x,t;v)$, where $f(x,t;v)\mathrm{d}x \mathrm{d}v$ is the number of hard rods in the volume element $\mathrm{d}x \mathrm{d}v$
 at time $t$. As conjectured by Percus \cite{P69} and later proved by Boldrighini, Dobrushin, and Soukhov \cite{bds}, under ballistic scaling the counting function becomes deterministic and is governed by a set of coupled nonlinear hyperbolic conservation laws of the form
 \begin{align}\label{2}
\partial_t f(x,t;v) +\partial_x \big((1 - \mathsf{r} \rho(x,t))^{-1} ( v -  \mathsf{r} \rho(x,t) u(x,t))f(x,t;v)\big) = 0.
\end{align}
Here   
\begin{align}\label{3}
 \rho(x,t) = \int_\mathbb{R}\mathrm{d}v f(x,t;v), \qquad u(x,t) =  \rho(x,t)^{-1} \int_\mathbb{R}\mathrm{d}v f(x,t;v) v
\end{align}
denote density and mean velocity at $(x,t)$. Further extensions and discussions can be found in \cite{BBDV22,D89}. Exact equilibrium spacetime  correlations were computed by Jespen \cite{J65}, Lebowitz, Percus 
 \cite{LP67}  and Lebowitz, Percus, Sykes \cite{LPS68}. Considerably later, fluctuation behavior was reanalysed  by Spohn{ \cite{spohn1982hydrodynamical,s}, Boldrighini, Wick \cite{boldrighini-wick}, and Presutti, Wick \cite {MR971036,MR968599}.
Note that, in the limit  $\mathsf{r} \to0$, Eq.\,\eqref{2} is linear and one recovers the ideal gas dynamics.

Hard rods have become a paradigmatic example for generalized hydrodynamics (GHD) and  serves as a schematic illustration of more intricate models as the Toda lattice \cite{D19,S19} and box-ball system \cite{F18,CS20}, see \cite{D19a,S21, BBDV22,ADDKS22}
 for lecture notes, special volumes and reviews.  In GHD one considers integrable many-body systems in one spatial dimension. 
Such systems have an infinite number of local conservation laws, which leads to the construction of generalized Gibbs ensembles (GGE).  Hydrodynamics then assumes that locally the system is in a GGE and, on the ballistic scale, its parameters evolve according to an equation of a similar structure as  \eqref{2}. The central microscopic input to hydrodynamic equations is the two-particle scattering shift,
which in general depends on the two incoming velocities.  Hard rods have the crucial simplification that the scattering shift is  $- \mathsf{r}$ independent of the incoming velocities.
Still, structural aspects of GHD are well illustrated by hard rods.

Rather than following the trajectories of mechanical particles, for hard rods the more efficient way is to turn to the motion of quasiparticles.  A quasiparticle maintains its own velocity. When hit a slower quasiparticle with length $\mathsf{r}$ the quasiparticle jumps by $\mathsf{r}$ and when it collides with a faster one, it will jump by $-\mathsf{r}$. Since the initial data are random, the jumps occur at random times. Note that in the limit $\mathsf{r} \to 0$ one recovers the straight line motion of an ideal gas. Quasiparticles also suggest a technique for proof: The initial configuration is compressed so 
to remove the hard core, still keeping the velocities. Then all particles evolve according to the free dynamics. The true hard rod configuration at time $t$ is obtained by the correspondingly reversed
expansion.  

In our article we report on a novel proof of \eqref{2} employing Chentsov Lantu\'{e}joul fields. We use the occasion to extend previous results in two directions.  Firstly we allow for a variable rod length, $\mathsf{r}_j$, attached to the quasiparticle with label $j$, implying that quasiparticles crossing it jump by  $\pm\mathsf{r}_j$. If the rod length would be assigned to particles, the model becomes intractable. 
The second extension is to allow for a negative rod length, $\mathsf{r}_j \in \mathbb{R}$. This appears unphysical at first look. However, as an example, the scattering shift of the Toda lattice is given by $2\log|v_1-v_2|$, $v_1,v_2$ the two incoming velocities. For small $|v_1-v_2|$ the shift is negative, Toda particles repel, and conventional hard rods are a good model. On the other hand for  $|v_1-v_2| > 1$ the shift turns positive which means that the two Toda particles cross each other before separating. 

The variable model with a finite number of positive rod sizes was first introduced by Aizenman, Lebowitz, and Marro \cite{ALM78},  thereby extending 
 the results for equilibrium spacetime correlations obtained in \cite{LPS68}. Very unexpectedly, hard rods with variable rod length appears as a phenomenological model  in an article by Cardy and Doyon \cite{cardy-doyon} in their study of $T\bar{T}$ deformations in relativistic and nonrelativistic integrable field theories.
 
 Compared to the seminal work \cite{bds}, we choose simpler initial conditions, namely a non-homogeneous Poisson process in $\mathbb{R}^3$, where each point 
 represents position, velocity, and rod length. Since reasonable initial measures converge in arbitrarily small macroscopic times to a Poisson 
 process \cite{s}, this is not such a restrictive assumption. As crucial advantage, the proof of the law of large numbers is transparent and almost immediate. In addition, hard rod joint positional fluctuations converge to a non-homogeneous Levy-Chentsov field. 

In Section~\ref{S2} we describe the model and state the results. In Section~\ref{S3} we introduce the ideal gas dynamics and its relation with the hard rod model. In Section~\ref{S4} we study the fields induced by line processes and their relation with hard rods, and prove the law of large numbers and the convergence of the positional fluctuations to the Levy Chentsov field. Section~\ref{S5} deals with the macroscopic setup and the proof of the macroscopic evolution theorem.\\\\
\textbf{Acknowledgments}. We thank Stefano Olla for motivating discussions. This article was partially written at the Mathematical Sciences Research Institute, during the semester Universality and Integrability in Random Matrix Theory and Interacting Particle Systems in 2021, with support of MSRI, Simon Foundation. PAF Acknowledges support from Conicet and Agencia of the Argentinian Science Ministry.

 \section{Summary of main results}
 \label{S2}

The segment $(y,y+\err)$ is called \emph{rod} of position $y$ and length $\err$, for the moment $r\ge 0$.  The 3-dimensional point $(y,v,\err)$ in $\RR^3$ represents a quasiparticle at position $y$ with an associated rod of length $\err$ traveling at speed~$v$. A \emph{hard rod} configuration $\sfY\subset \RR^3$ is a locally finite set of non-intersecting rods, that is, $\bigl(y,y+\err\bigr)\cap \bigl(\ty,\ty+\tr \bigr)=\emptyset$ for all $(y,v,\err),(\ty,\tv ,\tr )\in\sfY$. The evolution of a hard rod configuration is deterministic: each quasiparticle travels ballistically at its speed until a collision: if at time $t-$ the positions of quasiparticles $(y,v,\err)$ and $(\ty,\tv ,\tr )$ with $v>\tv $ satisfy $\ty=y+\err$, then, at time $t$ each quasiparticle is shifted in the direction of the other quasiparticle, interchanging order:
\begin{align}
  \begin{array}{ccc}
  \text{Before collision, at time $t-$} &&  \text{After collision, at time $t$}\\
  (y,v,\err), (\ty,\tv ,\tr ) && (y+\tr ,v,\err), (\ty-\err,\tv ,\tr )
  \end{array}\label{collision}  
\end{align}
After collision each quasiparticle continues at its speed until the next collision occurs.

The above description defines the hard rod exclusion interaction and dynamics when there is a finite number of initial nonnegative rods. A system including negative rods is defined in function of the ideal gas and a related field. Let $\sfX\subset\RR^3$ be a locally finite configuration belonging to a set $\kX$ defined later in \eqref{kX}. A point $(x,v,\err)\in\sfX$ represents an ideal gas particle at position $x$ traveling at speed $v$ of ``length'' $\err$, a real number; for the moment the length is just a mark attached to the particle. The \emph{ideal gas} configuration at time $t\in\RR$ is defined by
\begin{align}
  \label{tt1}
  T_t\sfX := \bigl\{(x+vt,v,\err): (x,v,\err)\in\sfX\bigr\}.
\end{align}
Let the field $\sfH[\sfX]$, be the function $\sfH:\RR^ 2\to\RR$ defined by
\begin{align}
  \label{H8}
  \sfH(t,x) := \sum_{(z,w,\err)\in\sfX} \err\,\bigl(\one\{z\ge0,\,z+w t<x\}- \one\{z<0,\,z+w t\ge x\}\bigr).
\end{align}
The notation $[\sfX]$ indicates dependency on $\sfX$, when not clear from the context. For each given $t$, the path $\bigl( \sfH(t,x) \bigr)_{x\in\RR}$ is a jump process with nonhomogeneous jumps, see Fig.\ref{f1}. Moving $t$, the increments of the path travel ballistically and interchange at collisions, see Fig.\ref{f2}. 
Define also the \emph{mass flow} $ \ttj[\sfX]$ as the function $\ttj:\RR^3\to\RR$ given by
\begin{align}
  \label{jH8}
  \ttj(x,v,t) = \sfH(t,x+vt) - \sfH(0,x).
\end{align}
this is the sum of the lengths crossing the segment $(x+vs)_{0\le s\le t}$ with speed smaller than $v$ minus the lengths crossing it with speed bigger than $v$. We are assuming that mass is the length of the particle.
Denote $D_a[\sfX]$ the \emph{dilation} with respect to the point $a$, defined by
\begin{align}
  \label{D0X}
  D_a(x)&:= x+\sfH(0,x)- \sfH(0,a),\\
  D_a\sfX&:= \bigl\{(D_a(x),v,\err):(x,v,\err)\in\sfX\bigr\},
\end{align}
see Fig.{}\ref{f1}. We observe that if all rod lengths are nonnegative then $D_a$ is invertible, otherwise $D_a$ may not be one-to-one. 
\begin{figure}[h!]
  \centering
  \includegraphics[width=.55\textwidth]{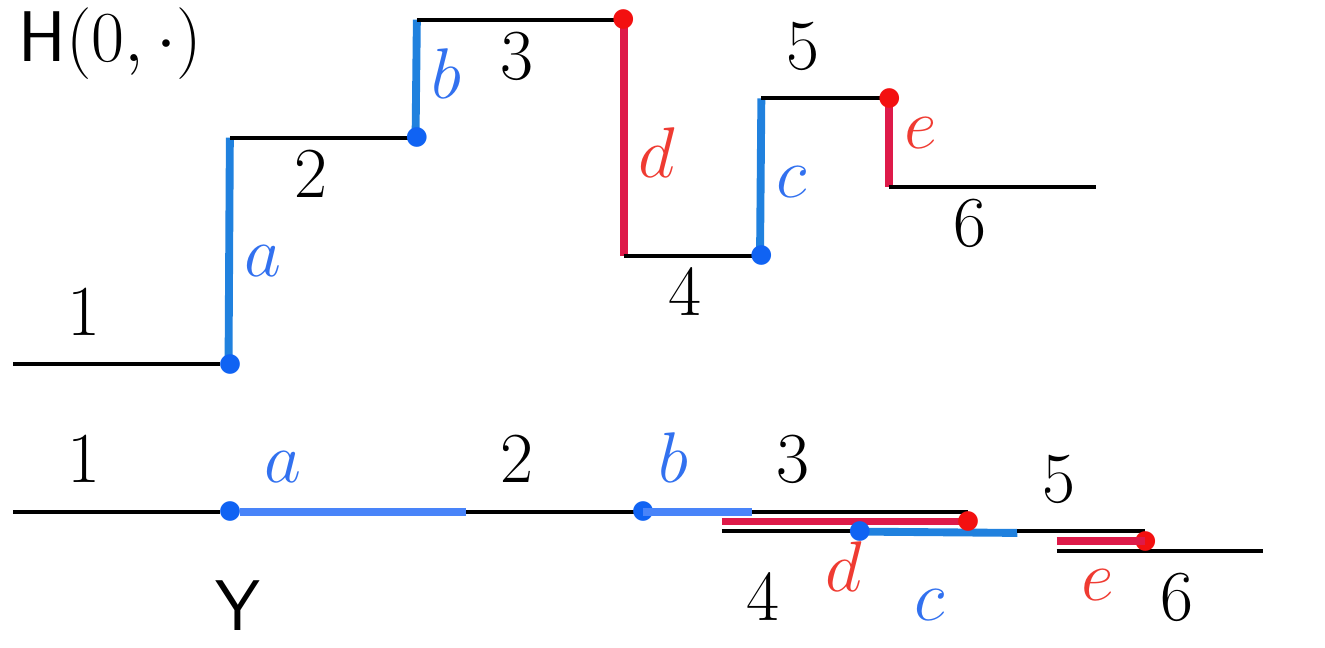}
  \caption{The upper path is a cut at time $0$ of the field $\sfH$ associated to the ideal gas configuration $\sfX$. The horizontal positions of dots are the position $x$ of particles $(x,v,\err)\in\sfX$. Blue and red vertical lines correspond to positive and negative $r$'s, respectively. Black horizontal segments correspond to interparticle segments in $\sfX$. The dots in the lower picture correspond to the position $y$ of quasiparticles $(y,v,\err)$ in $\sfY = D_0\sfX$. Red and blue segments in the graph of $\sfH$ are folded backwards and unfolded forward, respectively to obtain $\sfY$. All segments are contained in the same line, some pieces have been shifted down to distinguish overlapping rods. Numbers, letters and colors indicate the correspondence between segments in both figures. }
  \label{f1}
  \end{figure}%
Denote $\tty_{v,t}(x)[\sfX]$ the position at time $t$ of the quasiparticle associated to the ideal particle $(x,v,\err)$, defined by
\begin{align}
 \label{eq2}
  \tty_{v,t}(x)&:= D_0(x)+vt+\ttj(x,v,t)\\
 &= x+vt + \sfH(t,x+v t);            \label{yvt3}   
\end{align}
the identity \eqref{yvt3} is a consequence of  \eqref{jH8} and \eqref{D0X}.
Let $\sfY:=D_0\sfX$, and define the configuration $U_t\sfY[\sfX]$ by
\begin{align}
  \label{uty7}
  U_t\sfY:= \bigl\{( \tty_{v,t}(x),v,\err):(x,v,\err)\in\sfX\bigr\},
\end{align}
in particular $U_0\sfY=D_0\sfX=\sfY$. When all rods are nonnegative $(U_t)_{t\in\RR}$ is a group of operators: $U_0=$ Identity and $U_{t+s}=U_tU_s$. Otherwise there is no unique way to define $U_t\sfY$ as a function of $\sfY$. 
\begin{figure}[h!]
  \centering
  \includegraphics[width=.90\textwidth]{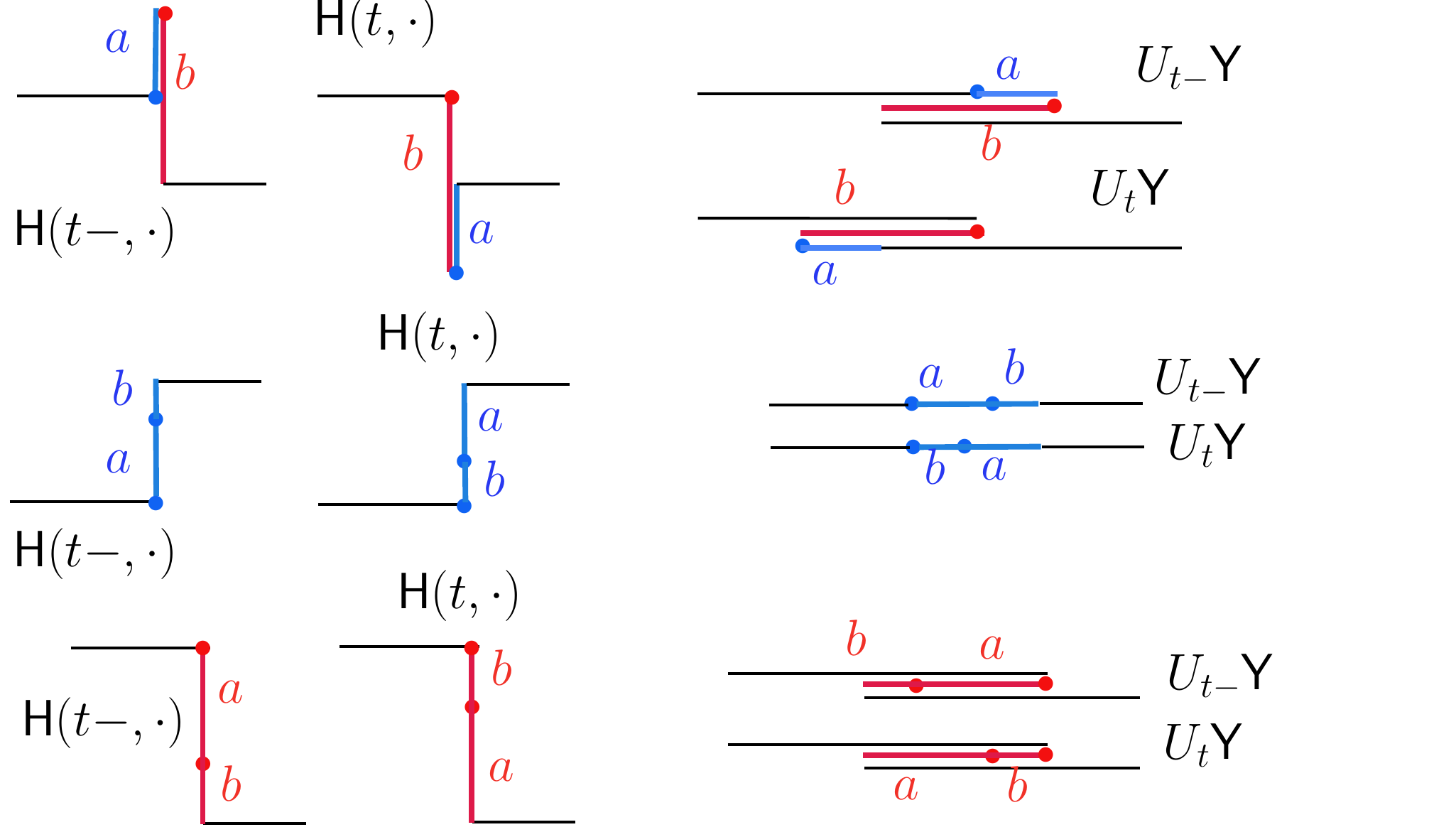}
  \caption{Effect of a collision in the field $\sfH[\sfX]$ and on the hard rod configuration $\sfY[\sfX]$. The left and middle columns picture the field $\sfH$ just before and after the collision, respectively. The right column shows the collision in the hard rod picture. In the first row we picture a positive rod (blue) and a negative one (red), before and after a collision. Second and third rows: collision between rods of the same sign. The labeled $a$ rod is faster than the labeled $b$ one  }
  \label{f2}
\end{figure}

Let $\euF$ be a set of intensities $f:\RR^3\to\RR_{\ge0}$ defined later in \eqref{sigma>0}.
Let $\sfX^\eps$ be a Poisson process on $\RR^3$ with intensity $\eps^{-1}f$ and call $\PP$ and $\EE$ the probability and expectation associated to the family $(\sfX^\eps)_{\eps>0}$.
Define the rescaled field $\sfH^\eps[\sfX^\eps]$ by
\begin{align*}
 \sfH^\eps(t,x):= \eps\sum_{(z,w,\err)\in\sfX^\eps} \err\bigl(\one\{z\ge0,\,z+w t<x\}- \one\{z<0,\,z+w t\ge x\}\bigr).
\end{align*}
The rescaled quantities are the functions of $\sfX^\eps$ defined by
\begin{align}
  \ttj^\eps(x,v,t) &:= \sfH^\eps(t,x+vt) - \sfH^\eps(0,x),\\
  D_a^\eps(x) &:= x+\sfH^\eps(0,x)-\sfH^\eps(0,a),\\
  D_a^\eps\sfX^\eps &:=\bigl\{\bigl(D^\eps_a(x),v,\err\bigr):(x,v,\err)\in\sfX^\eps\bigr\}, \label{yet1}\\
  \tty^\eps_{v,t}(x)&:=  x+vt+ \sfH^\eps(t,x+vt),\\
 U_t^\eps\sfY^\eps&:= \bigl\{(\tty^\eps_{v,t}(x),v,\err): (x,v,\err)\in\sfX^\eps\bigr\}.
\end{align}
Their expectations are denoted by
\begin{align}
   H(t,x) & :=\EE \sfH^\eps(t,x)
            = \iiint f(z,w,\err) \,\err\bigl(\one\{z>0,\,z+w t<x\}\label{josefa16}\\
          & \qquad\qquad\qquad\qquad\qquad - \one\{z<0,\,z+w t>x\}\bigr)\,dzdwd\err.\notag\\
 j(x,v,t)  &:= \EE\ttj^\eps(x,v,t) = H(t,x+vt)- H(0,x), \label{josefa13}\\
y_{v,t}(x) &:= \EE\tty^\eps_{v,t}(x) =x+vt + H(t,x+vt);
\end{align}
the second identity in \eqref{josefa16} follows from Campbell's theorem.

The rescaled random quantities converge to their expectations, a classical result for Poisson processes, shown in the next theorem.
\begin{theorem}[\rm Law of large numbers] 
  \label{tlln1}
  Assume $f\in\euF$. Then,
  for each $x,v,t\in\RR$ the following limits hold $\PP$-a.s.:
  \begin{align}
  \lim_{\eps\to0}   \sfH^\eps(t,x) = H(t,x),\quad
    \lim_{\eps\to0}   \ttj^\eps(x,v,t)  =  j(x,v,t),\quad
   \lim_{\eps\to0}  \tty^\eps_{v,t}(x)  =  y_{v,t}(x). \notag
  \end{align}
\end{theorem}
Define 
macroscopic dilation function $\D_a[f]$ by
\begin{align}
  \label{d0f}
  \D_a(x):=x+\int_a^x   \Bigl(\iint \err \,f(z,v,\err) dvd\err\Bigr)\,dz,
\end{align}
The set $\euF$ includes the condition $\sigma_f(z)>0$ for all $z$, implying that $\D_a[f]$ is invertible. Denote $\C_a[f]$ its inverse and define
\begin{align}
  \ttu^\eps_{v,t}(q)&:= \tty_{v,t}^\eps(\C_0(q)) = q+vt+\ttj^\eps(\C_0(q),v,t),\\
   \label{uvti}
  u_{v,t} (q)&:= y_{v,t}\bigl(\C_0(q)\bigr)= q+vt+j(\C_0(q),v,t).
\end{align}
\begin{corollary}[\rm Quasiparticle law of large numbers]
  \label{llnut4}
 Let $f\in\euF$. Then
 \begin{align}
  \label{test-limit2}
  \lim_{\eps\to0}\, \ttu^\eps_{v,t}(q) = u_{v,t} (q),\quad \PP\text{-a.s.},\quad q,v,t\in\RR.
 \end{align}
 If $f(x,v,\err)=0$ for $\err<0$, then the microscopic dilation $D^\eps_0$ is invertible and for all $q,v,t\in\RR$, 
\begin{align}
  \label{tut7}
  \tilde\ttu^\eps_{v,t}(q):=\tty^\eps_{v,t}(C^\eps_0(q))\; \mathop{\longrightarrow}_{\eps\to0}\; y_{v,t}(\C_0(q))= u_{v,t} (q),\quad \PP\text{-a.s.}.
  \end{align}
\end{corollary}
The random initial point $\ttu^\eps_{v,0}(q)=D_0^\eps(\C_0(q))$ converges to $\D_0(\C_0(q))=q$, as $\eps\to0$, by \eqref{eq2} and Theorem \ref{tlln1}. This choice permits us to compute the limit when the microscopic dilation $D^\eps_a(x)$ is not invertible, due to the presence of negative rods. When all rods are positive, we can start with a deterministic $q$, to get the same limit \eqref{tut7}.

We now discuss the convergence of the length fields. Along the paper we denote by $\varphi:\RR^3\to\RR$ a generic test function satisfying
\[
  \int f(x,v,\err)\, (v^ 2+\err^2+1)\,\varphi(x,v,\err)\,dxdvd\err<\infty.
\]
Define  $\ttK^\eps_t[\sfX^\eps]$, the  \emph{empirical length measure} at time $t$, acting on test functions $\varphi$ by
\begin{align}
  \ttK^\eps_t \varphi
  &:= \;\eps\,\sum_{(y,v,\err)\,\in\, U_t\sfY^\eps}\,r\,\varphi(y,v,\err)
    = \;\eps\,\sum_{(x,v,\err)\,\in\, \sfX^\eps}  \err\,\varphi(\tty^\eps_{v,t}(x),v,\err),
\end{align}
by \eqref{uty7}. Let $\kappa_t$ be the length measure defined by 
\begin{align}
  \label{kt14}
 \kappa_t\varphi :=  \iiint f(x,v,\err)\,r\,\varphi(y_{v,t}(x),v,\err)\,dxdvd\err, \quad t\in \RR.
\end{align}
 The next result is a particular case of Theorem \ref{lln}, proved in Section~\ref{S4}.
\begin{theorem}[\rm Law of large numbers for the empirical measure]
  \label{llnhr}
Let $f\in \euF$, then 
 \begin{align}
    \lim_{\eps\to0} \ttK^\eps_t \varphi 
   = \kappa_t \varphi, \qquad\PP\text{-a.s. and in $L_1$},
  \end{align}
  for all $t\in\RR$ and test function $\varphi$.
\end{theorem}
 Define the random field $\eta^\eps[\sfX^\eps]$ by
\begin{align}
    \eta^\eps(t,z) := \frac1{\sqrt\eps} \Bigl(\sfH^\eps(t,z)-H(t,z)\Bigr).
\end{align}
The positional hard rod fluctuations satisfy
\begin{align}
 \frac1{\sqrt\eps} \Bigl(\tty^\eps_{v,t}(x) -y_{v,t}(x)\Bigr) =   \eta^\eps(t,x+vt).
\end{align}

A  \emph{Levy Chentsov field} associated to a distance $\dd$ in $\RR^2$ is a centered Gaussian process $\eta:\RR^2\to\RR$ with $\eta(\bo)=0$  and covariances
  \begin{align}
    \Cov \bigl(\eta(\ba),\eta(\bb)\bigr) =
\frac12\bigl(\dd(\bo,\ba)+\dd(\bo,\bb)-\dd(\ba,\bb)\bigr),\quad \ba,\bb\in\RR^2.
  \end{align}
  Here $\bo$ is the origin of $\RR^2$. Denote $[\ba,\bb]$ the one dimensional segment contained in $\RR^2$ with extremes $\ba$ and $\bb$. The following theorem is proved in Section~\ref{S4}.
  
\begin{theorem}[\rm Positional fluctuations]
  \label{crf}
  Let $f\in\euF$. As $\eps\to0$, the finite dimensional distributions of the random field $\eta^\eps$ 
  converge in distribution to those of the Levy Chentsov field $\eta:\RR^2\to\RR$, associated to the distance
\begin{align}
  \dd(\ba,\bb) := \iiint f(x,v,\err)\, \err^2 \,
  \one\{(x+vt)_{t\in\RR} \text{ intersects } [\ba,\bb]\}\, dxdvd\err.
\end{align}
In particular, the positional fluctuation random field converges to the Levy Chentsov field:
\begin{align}
  \lim_{\eps\to0}   \Bigl(\frac1{\sqrt\eps} \bigl(\tty^\eps_{v,t}(x) -y_{v,t}(x)\bigr) \Bigr)_{(x,v,t)\in N} \eqlaw  \Bigl(
  \eta(x,x+vt)  \Bigr)_{(x,v,t)\in N}
\end{align}
for any finite set $N\subset\RR^3$. 
\end{theorem}

\paragraph{Macroscopic evolution} Let $f\in\euF$, and  define
\begin{align}
  g(q,v,\err) &:= f(\C_0(q),v,\err) \frac{d}{dq}\C_0(q),\\
  g_t (q,v,\err) &:= f (y_{v,t}^{-1}(q),v,\err)\frac{d}{dq}y_{v,t}^{-1}(q).
\end{align}
Notice that $g_0=g$ and that $g_t$  satisfies
\begin{align}
  \label{gtf9}
   \kappa_t\varphi=\iiint g_t(q,v,\err)\,\err\,\varphi(q,v,\err)\,dqdvd\err.
\end{align}
The following theorem, proved in  Section \ref{S5}, describes the equation for the macroscopic hard rod evolution $(g_t)_{t\in\RR}$. 
\begin{theorem}
  \label{cauchy-theorem}
 Let $f\in \euF$ such that $\sigma_f(z)>0$ for all $z$. Then $g_t$ is the unique solution of the Cauchy problem 
\begin{align}\label{cauchy-g}
  \partial_t g_t(q,v,\err) &=- \partial_q\bigl(g_t(q,v,\err)\,v^{\eff}(q,v,t) \bigr),\\[1mm]
    v^{\eff}(q,v,t)&= v+\frac{\iint\,\err\,(v-w)\,g_t(q,w,\err)\,dw\,d\err}
  {1 - \iint\,\err\, g_t(q,w,\err)\,dw\,d\err},\\
   g_0(q,v,\err) &= f (y_{v,t}^{-1}(q),v,\err)\frac{d}{dq}y_{v,t}^{-1}(q).
\end{align}
Furthermore, the limiting trajectory $u_{v,t}(q)$, considered as a function of $t$, satisfies the Cauchy problem
\begin{align}
    \frac{\partial}{\partial t} u_{v,t}(q) &= v^{\eff}(u_{v,t}(q),v,t)\label{cauchy-u}\\
     u_{v,0}(q) &= q
  \end{align}
\end{theorem}

\section{Line processes and random fields}
\label{S3}

\paragraph{Line measures and fields}
The \emph{ideal gas representation of lines} is given by 
the map
\begin{align}
  \label{x-lines}
  (x,v)\mapsto (t,x+vt)_{t\in \RR}.
\end{align}
This map is a bijection between $\RR^2$ and the space of straight lines contained in $\RR^2$, excluding the lines perpendicular to the $t$ axis $(t,x)_{x\in\RR}$. 
An ideal gas trajectory is seen as a line contained in $\RR^2$ with an \emph{orientation} from past to future time. 
We think of a point $(x,v,\err)$ in $\RR^3$ as the line $(t,x+vt)_{t\in \RR}$ associated with the \emph{mark} $\err\in\RR$. 

For time-space points $\ba,\bb\in\RR^2$, denote $[\ba,\bb]:= \{(1-u)\ba+u\bb:u\in[0,1]\}$, the (one-dimensional) segment contained in $\RR^2$ with extremes $\ba$ and $\bb$ and \emph{oriented}  $\ba\to\bb$; the segment $[\bb,\ba]$ occupies the same set of points as $[\ba,\bb]$ but has opposite orientation. Denote $\ba=(t_\ba,x_\ba)$ and $\bb=(t_\bb,x_\bb)$. The ``speed'' $v_{\ba\bb}$ of the segment is defined by $v_{\ba\bb}:= \frac{x_\bb-x_\ba}{t_\bb-t_\ba}$, if $t_\ba\neq t_\bb$ and by $v_{\ba\bb}= \pm \infty$, according to the sign of $x_\bb-x_\ba$, if $t_\ba=t_\bb$.

Denote $\bab$ the set of marked lines intersecting $[\ba,\bb]$:
\begin{gather}
  \label{ab}
  \bab:=\bigl\{(x,v,r): (t,x+vt)_{t\in \RR} \text{ intersects }[\ba,\bb]\bigr\};\\
\label{abpm}
\bab_- := \{ (x,v,r)\in\bab: v>v_{\ba\bb}\}, \quad  \bab_+ := \{ (x,v,r)\in\bab: v<v_{\ba\bb}\};
\end{gather}
Looking from $\ba$ to $\bb$, the set  $\bab_-$ contains the lines crossing the segment from right to left and $\bab_+$ contains those crossing from left to right.

Let $\euM$ be the set of measures $\mu$ on $\RR^3$ with the Borel sigma algebra $\euB(\RR^3)$, satisfying
\begin{align}
  \label{euM}
 \sup_{-\infty<a<b<\infty} \frac1{b-a} \int_{[a,b]\times \RR^2} \mu(dx,dv,d\err) \, (v^2+ \err^2+1)< \text{Constant}.
\end{align}
The second moment conditions will be necessary to construct infinite volume fields and hard rod evolutions. In particular, the space marginal of $\mu$ is sigma finite.

For $\mu\in\euM$, define the signed measure $\mu_1$ and the measure $\mu_2$ by
\begin{align}
  \label{mu1}
 \mu_1 (dx,dv,d\err) &:= \err\,\mu(dx,dv,d\err),\\
  \mu_2(dx,dv,d\err)&:= \err^2\,\mu(dx,dv,d\err). \label{mu2}
\end{align}
Define the field $H=H[\mu]$ as the function $H:\RR^2\to\RR$, given by
\begin{align}
  \label{hmu}
  H(\ba) := \mu_1 (\boa_+)-\mu_1 (\boa_-),
\end{align}
where $\bo$ is the origin of $\RR^2$.

\paragraph{Chentsov Lantuéjoul fields}

To each marked line $(x,v,\err)$ associate the function $\sfH_{(x,v,\err)}:\RR^2\to\RR$, defined by
\begin{align}
  \sfH_{(x,v,\err)}(\ba):=
  \begin{cases}
    0& \text{if } (x,v)\notin\boa\\
    \err&\text{if }(x,v)\in\boa,\,\text{ and }x>0,\\
    -\err&\text{if }(x,v)\in\boa,\,\text{ and }x<0.
  \end{cases}
\end{align}
Considering that the line $(x,v)$ splits $\RR^2$ into two semi planes, the function  $\sfH_{(x,v,\err)}$ gives height $0$ to the  points of the semi plane containing the origin, height $\err$ to the points of the other semi plane, if the line cross the $x$ axis at a positive value, and height $-\err$ if the line crosses the $x$ axis at a negative value. See Fig.~\ref{lcs11}.

\begin{figure}[h]
  \centering
    \includegraphics[width=.60\textwidth]{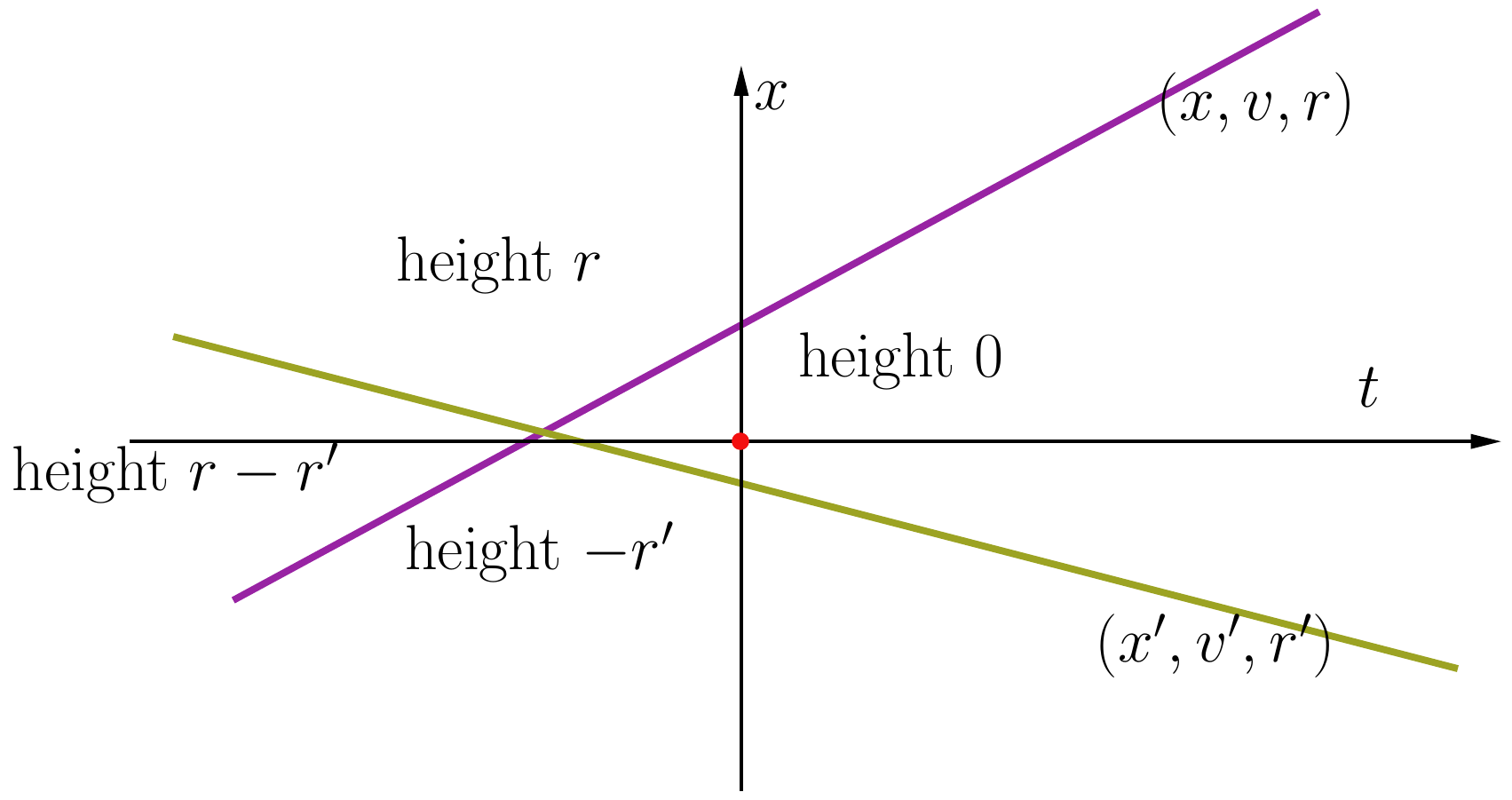}
    \caption{The marked lines $(x,v,\err)$ and $(x',v',\err')$ partition the time-space plane in 4 cones, each one at a constant height for the field $\sfH_{(x,v,\err)}+\sfH_{(x',v' ,\err')}$, as indicated.  \label{lcs11}} 
  \end{figure}
Given a configuration $\sfX\in\kX$, defined in \eqref{kX},  define the field $\sfH=\sfH[\sfX]$ by
\begin{align}
  \label{xi45}
  \sfH(\ba) := \sum_{(x,v,\err)\in\sfX} \sfH_{(x,v,\err)}(\ba), \quad \ba\in\RR^2.
\end{align}
The sum is finite, as it only collects the contribution of the lines intersecting the segment $[\bo,\ba]$, a property satisfied by $\sfX\in\kX$. In particular, $\sfH(\bo)=0$ and  for $\ba,\bb\in\RR^2$ we have
\begin{align}
  \label{xiab}
  \sfH(\bb)-  \sfH(\ba) &=   \sum_{(x,v,\err)\in\sfX}  r\,\Bigl( 1\{(x,v)\in\bab_-\}- 1\{(x,v)\in\bab_+\}\Bigr).
\end{align}
Denote $\ttK[\sfX]$ by
\begin{align}
  \label{ttK}
    \ttK(A) := \sum_{(x,v,\err)\in \sfX}\err\,\one\{(x,v,\err)\in A\},\qquad A\in\euB(\RR^3),
  \end{align}
  the length empirical measure associated to $\sfX$. We have
  \begin{align}
    \label{Hoa9}
    \sfH(\ba) = \ttK(\bob_+)-\ttK(\bob_-).
  \end{align}
Let $\mu\in\euM$ and let $\sfX$ be a Poisson process in $\RR^3$, with intensity measure $\mu$. $\sfX$ can be seen as a marked Poisson line process, \cite{kingman}. The process $\sfH[\sfX]$ is called \emph{Chentsov Lantuéjoul field} with control measure $\mu$. The name comes from the Chentsov construction of the L\'evy's \cite{MR0029120} Brownian process with several parameters, also called Levy Chentsov field, defined later, and from Lantuéjoul \cite{lantuejoul-book}, who introduced a random field built from a Poisson line process with random marks $\err$ in $\{-1,+1\}$.

\begin{figure}[h]
  \centering
  \includegraphics[width=.45\textwidth]{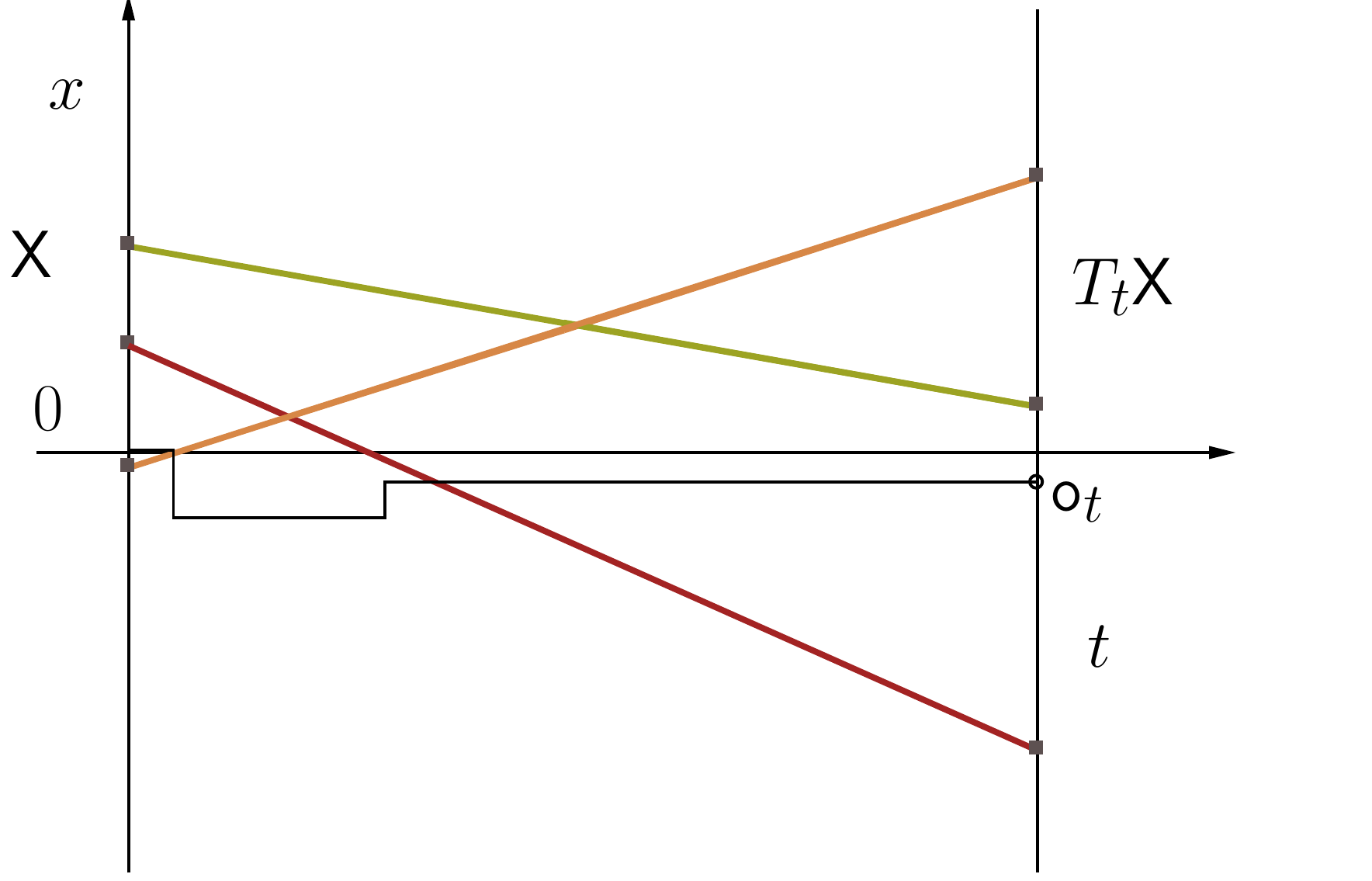}
 \includegraphics[width=.45\textwidth]{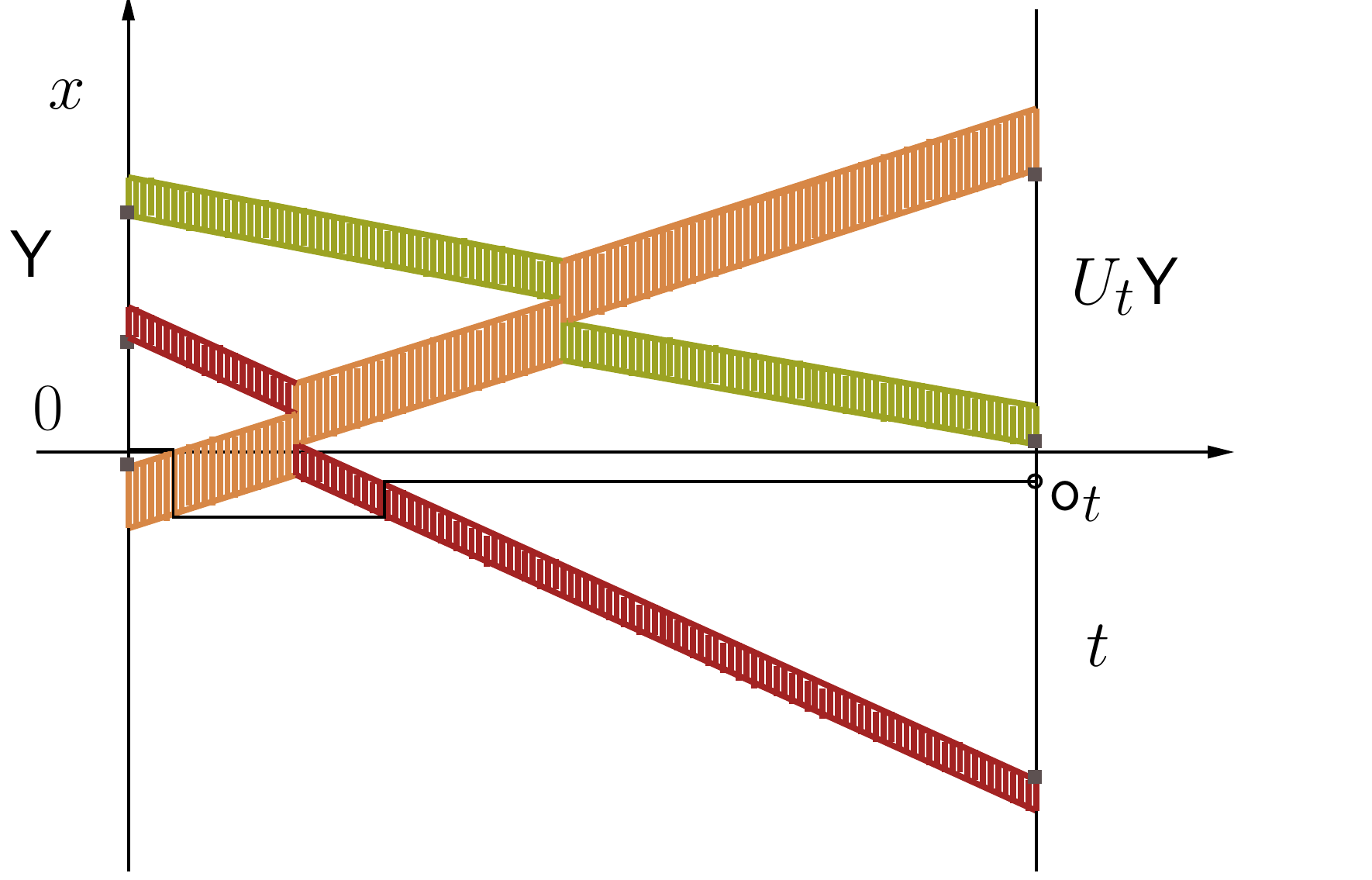}
  \caption{To the left, evolution of ideal particles and the position $\tto_t$. To the right, evolution of associated quasiparticles. Interpreting the colored trajectories as vertical steps, the right figure can be seen as a three dimensional perspective of the surface $\sfH$.}
  \label{yxvt3-yxvt2}
\end{figure}

From \eqref{xiab} and \eqref{hmu} we conclude that if $\sfX$ is a Poisson process with intensity $\mu$, $\sfH=\sfH[\sfX]$ and $H=H[\mu]$, then
 \begin{align}
   \EE \sfH(\ba) = H(\ba),\quad \ba\in\RR^2. \label{hach}
 \end{align}

 \paragraph{Law of large numbers}
Let $\mu\in\euM$ and $\mu_1 $ as in \eqref{mu1}. 
We construct a family of Poisson processes $(\sfX^\eps)_{\eps>0}$ as projections of a unique Poisson process $\bsfX$ in $\RR^3\times \RR$, with intensity measure $\mu(dx,dv,d\err)\, dz$, by defining
\begin{align}
  \label{xxe3}
  \sfX^\eps:=\{(x,v,\err):(x,v,\err,z)\in \bsfX,\, 0<z<\eps^{-1}\}.
\end{align}
For each $\eps>0$, $\sfX^\eps$ is a Poisson process on $\RR^3$ with intensity measure $\eps^{-1}\mu$. 
Let $\PP$ and $\EE$ be the probability and expectation associated to the process $\bsfX$. Denote
$\sfH^\eps:=\eps\sfH[\sfX^\eps]$ the rescaled Chentsov Lantuéjoul field given by
\begin{align}
  \sfH^\eps(\ba) &:= \eps\sum_{(x,v,\err)\in\sfX^\eps}\sfH_{(x,v,\err)}(\ba)\notag\\
  &= \ttK^\eps(\boa_+)-\ttK^\eps(\boa_-), \label{HK6}
\end{align}
by \eqref{Hoa9}, where $\ttK^\eps(A) = \eps\ttK(A)$.
\begin{proposition}
  \label{llnCLS}
Let $\mu\in\euM$ and $(\sfX^\eps)_{\eps>0}$ be a family of Poisson processes, as defined in  \eqref{xxe3}. Then,
\begin{align}
  \label{lln8}
    \lim_{\eps\to0} \sfH^\eps(\ba) = H(\ba),\qquad \PP\text{-a.s.},\quad \ba\in\RR^2,  \end{align}
  where $H=H[\mu]$ was defined in \eqref{hmu}.
\end{proposition}
\begin{proof}
For integer $\eps^{-1}$, $\sfX^\eps$ is the union of $\eps^{-1}$ processes
\begin{align}
  \label{xie}
  \sfX^\eps = \bigcup_{i=1}^{\eps^{-1}} \sfX_i,\qquad \sfX_i := \{(x,v,\err):(x,v,\err,z)\in \bsfX,\, i-1\le z<i\}.
\end{align}
By definition, $\sfX_i$ are iid Poisson processes of intensity measure $\mu$. Hence, denoting $\ttK_i= \ttK[\sfX_i]$, we have 
\begin{align}
  \label{llni}
 \sfH^\eps(\ba) &= \eps\sum_{i=1}^{\eps^{-1}} \bigl(\ttK_i(\boa_+)-\ttK_i(\boa_-)\bigr)\;\;\mathop{\longrightarrow}_{\eps\to0}\;\;  \mu_1 (\boa_+)-\mu_1 (\boa_-),\;\;\text{a.s.},
  \end{align}
  by the law of large numbers for iid random variables, because $\EE\bigl(\ttK_i(\boa_-)-\ttK_i(\boa_-)\bigr)=\mu_1 (\boa_+)-\mu_1 (\boa_-)=H(\ba)$.
\end{proof}

\paragraph{Line white noise and Levy Chentsov field}

Let $\mu\in\euM$, and  $\mu_2$ as in \eqref{mu2}.
Let \emph{marked line white noise} with \emph{control measure} $\mu_2$ be the random signed measure $\omega$ on $\euB(\RR^3)$ satisfying 
\begin{itemize}
\item[a)] for any Borel set $A$ with $\mu_2(A)<\infty$, the random variable $\omega(A)$ has centered Gaussian distribution with variance $\mu_2(A)$.

\item[b)] $\Cov (\omega(A),\omega(B))= \mu_2(A\cap B)$.
\end{itemize}

Following Chentsov \cite{chentsov1957levy}, we construct the \emph{Levy Chentsov field} $\eta=(\eta(\ba))_{\ba\in\RR^2}$ as a function of the white noise $\omega$. Recall that $\boa$ is the set of lines intersecting the segment $[\bo,\ba]$ and define 
\begin{align}
  \eta(\ba)&:= \omega\bigl(\boa\bigr),\qquad \ba\in\RR^2. \notag
\end{align}
By definition, the field $\eta$ is centered Gaussian with 
covariances
\begin{align*}
 \Cov\bigl(\eta(\ba),\eta(\bb)\bigr)
  &=\; \Cov\bigl( \omega(\boa), \omega(\bob)\bigr)=\;\mu_2(\boa\cap\bob)\\
  &\;=\;\tfrac12\bigl(\mu_2(\boa)+\mu_2(\bob)-\mu_2(\bab)\bigr).
  \end{align*}
To check the last identity use $\bab\, \dot\cup\, (\boa\cap\bob) = \boa\cup\bob$. We say that $\eta$ is a Levy-Chentsov field associated to the distance $\dd$ defined by $\dd(\ba,\bb):= \mu_2(\bab)$.

The height of the field $\eta$ along any given line contained in $\RR^2$ is Brownian motion:
\begin{lemma}[\rm Time changed Brownian motion]
Let $\mu\in\euM$ and $\eta$ be the Levy Chentsov field for the distance $\dd(\ba,\bb)=\mu_2(\bab)$. The processes
$ \bigl(\eta(t,x+vt)\bigr)_{t\in\RR}$ and $(\eta(t,x))_{x\in\RR}$
satisfy 
\begin{align}
  \bigl(\eta(t,x+vt)-\eta(0,x)\bigr)_{t\in\RR} 
  & \;\eqlaw \;
    \bigl(B(\mu_2(\overline{\ba_0\ba_t})\bigr)_{t\in\RR}\;\;,
    &\ba_s&:=(s,x+vs) 
    ;\notag\\
  \bigl(\eta(t,x)-\eta(t,0)\bigr)_{x\in\RR}
  & \;\eqlaw   \; \
    \bigl(B(\mu_2(\overline{\bb_0\bb_x})\bigr)_{x\in\RR}\;\;,
    &\bb_x&:=(t,x). \notag
\end{align}
Here $(B(\tau))_{\tau\in\RR}$ is standard one-dimensional Brownian motion.
\end{lemma}
\begin{proof}
By the definition of white noise, both processes are Gaussian, have independent increments and their variances are given respectively by
\begin{align*}
  \Var \bigl(\eta(t,x+vt) - \eta(s,x+vs)\bigr)
  &= \mu_2(\overline{\ba_s\ba_t})
    = \Var B(\mu_2(\overline{\ba_s\ba_t})),\notag\\
  \Var\bigl(\eta(t,x)-\eta(t,z)\bigr)
  &= \mu_2(\overline{\bb_x\bb_z})
    = \Var B(\mu_2(\overline{\bb_x\bb_z})). \tag*{\qedhere}
\end{align*}
\end{proof}

\paragraph{Field fluctuations}
Under diffusive rescaling, a Chentsov Lantuéjoul field converges to a Levy Chentsov field. 
Define $\eta^\eps=\eta^\eps[\sfX^\eps]$, by
\begin{align}
  \label{et5}
  \eta^\eps(\bb) := \frac{\sfH^\eps(\bb)- \EE \sfH^\eps(\bb)}{\sqrt\eps}, \quad \bb\in\RR^2.
\end{align}
where $\sfH^\eps$ was defined after \eqref{HK6}. 
\begin{proposition}[\rm Convergence to Levy Chentsov field]
  \label{sef1}
  Let $\mu$ and $\sfX^\eps$ as in Proposition \ref{llnCLS}. The finite dimensional distributions of the field $\eta^\eps$ converge to those of the field $\eta$: for any finite set $N\subset \RR^2$  
  \begin{align}
    \lim_{\eps\to0}  (\eta^\eps(\ba))_{\ba\in N} \eqlaw  (\eta(\ba))_{\ba\in N},
  \end{align}
  where $\eta$ is the Levy Chentsov field associated with the distance $\dd(\ba,\bb)= \mu_2(\bab)$.
\end{proposition}
\begin{proof}
Since $\sfH^\eps(\ba)=\ttK^\eps(\boa_+)-\ttK^\eps(\boa_-)$ and the sets $\boa_-$ and $\boa_+$ are disjoint, it suffices to show that 
  \begin{align}
    \label{clt3}
    \frac1{\sqrt\eps}\bigl( \ttK^\eps(A) - \mu_1(A)\Bigr) \eqlaw \omega(A),    \qquad A\in\euB(\RR^3),
  \end{align}
where $\omega$ is white noise in $\euB(\RR^3)$ with control measure $\mu_2$, in particular $\omega(A)$ is centered Gaussian with variance $\mu_2(A)$. To see it,   
recall the decomposition \eqref{xie} and for integer $\eps^{-1}$ write $\ttK^\eps(A)[\sfX^\eps]=\sum_{i=1}^{\eps^{-1}} \ttK(A)[\sfX_i]$, so that $\ttK^\eps(A)$ is a sum of iid random variables with mean $\mu_1(A)$, divided by the number of terms. The convergence \eqref{clt3} follows from the central limit theorem and, by Campbell's theorem, the covariances are
\begin{align}
 \Cov (\ttK^\eps(A),\ttK^\eps(B)) &= \mu_2(A\cap B),\quad A,B\in \euB(\RR^3),\;\; \eps>0. \notag\tag*{\qedhere}
\end{align}
\end{proof}

\section{Hard rod limit theorems}
\label{S4}

When the rod lengths are all positive, the hard rod dynamics can be defined as a group of operators $(U_t)_{t\in\RR}$,
where $U_t\sfY$ is the hard rod configuration at time $t$. When negative lengths are allowed, $U_t\sfY$ can still be defined as a family of configurations, but $U_t$ cannot be seen as an operator because $U_t\sfY$ is not a function of $\sfY$.  However, one can still obtain limit theorems when the intensity measure $\mu$ satisfies some conditions. 

The ideal gas evolution operator $T_t$ is defined by
\begin{align}
T_t \sfX  &\coloneqq \{(x+vt,v,\err)\in \RR^3:(x,v,\err)\in\sfX\}.
\end{align}
$T_t$ is a group: $T_0=$ Identity and $T_{t+s}=T_tT_s$.
The set of configurations $\sfX $ with finite absolute mass flows is defined by
\begin{align}
  \label{kX}
  \kX &:= \Bigl\{\sfX\subset \RR^3 : \sum_{(z,w,\err)\in\sfX}(|\err|+1) \bigl(\one\{z<x,z+wt\ge x+vt\}\\
  &\qquad\qquad\qquad\qquad+\one\{z\ge x,z+wt<x+vt\}\bigr)<  \infty ;\;  x,v,t \in\RR \Bigr\}.\notag
\end{align}
In particular, if $\sfX\in\kX$, then $\sfX$ is locally finite and $|m_a^b(\sfX)|<\infty$, $a<b\in\RR$.

\paragraph{Positive lengths}
We consider some properties of the group $(U_t)_{t\in\RR}$ in the positive case. Denote the set of nonnegative length configurations by
\begin{align}
  \label{>0}
\kXa:=\{\sfX\in\kX: \err\ge0 \text{ for all }(x,v,\err)\in\sfX\},
\end{align}
the corresponding hard rod configuration set by
\begin{align}
  \notag
  \kYa:=\{\sfY\in \kXa: (y,y+\err)\cap (\ty,\ty+\tr )=\emptyset,\; (y,v,\err), (\ty,\tv ,\tr )\in \sfY\},
\end{align}
and the set of configuration with no rod containing $a$ is denoted by
\begin{align}
  \label{kYa}
  \kY_a:=\{\sfY\in \kYa: a\notin (y,y+\err),\; (y,v,\err)\in \sfY\}.
\end{align}
Notice that $\kY_a$ include configurations having a hard rod (with left extreme) at $a$. 
The dilation with respect to the point $a$, defined in \eqref{D0X}, now are labeled with the configuration and can be expressed by
\begin{align}
  \label{D0X7}
  D_{\sfX,a}(x)&= x+m_a^b(\sfX),\\
  D_a\sfX&= \bigl\{(D_{\sfX,a}(x),v,\err):(x,v,\err)\in\sfX\bigr\},
\end{align}
where 
the \emph{signed mass} of $\sfX $ between real points $a$ and $b$ is 
\begin{align}
 m_a^b(\sfX ) :=
  \begin{cases}
    \sum_{(x,v,\err)\in\sfX}\; \err\,\one\{a\le x<b\}&\text{if } a<b\\
    - \sum_{(x,v,\err)\in\sfX}\; \err\,\one\{b\le x<a\}&\text{if } a>b.
  \end{cases}
\end{align}
Under \eqref{>0}  
the map $D_a:\kX\to\kY_a$ is a bijection and its inverse is the \emph{contraction from $a$} map 
$C_a:\kY_a\to\kX$, given by
\begin{align}
   C_{\sfY,a}(y) &:= y - m_a^y(\sfY),\\
    C_a\sfY &:= \bigl\{\bigl(C_{\sfY,a}(y),v,\err\bigr):(y,v,\err) \in \sfY\bigr\}.\label{CaY}
\end{align}
If $\sfY\in \kY\setminus \kYa$, then the map $D_{\sfY,a}$ is not necessarily one-to-one, so the inverse is not well defined. 

Let $\sfY \in\kY_y$ 
and denote $\ttu_{v,t}(y)[\sfY]$ the position at time $t$ of a quasiparticle with arbitrary finite length inserted at time zero in $y$, with speed $v$, defined by
\begin{align}
  \label{test1}
  \ttu_{v,t}(y)&:= y+vt + \ttj(y,v,t)[C_y\sfY].\qquad\sfY \in\kY_y,
\end{align}
where \emph{net mass flow} $\ttj(x,v,t)[\sfX]$ was defined in \eqref{jH8}. 
\begin{figure}
  \centering
  \includegraphics[width=.60\textwidth]{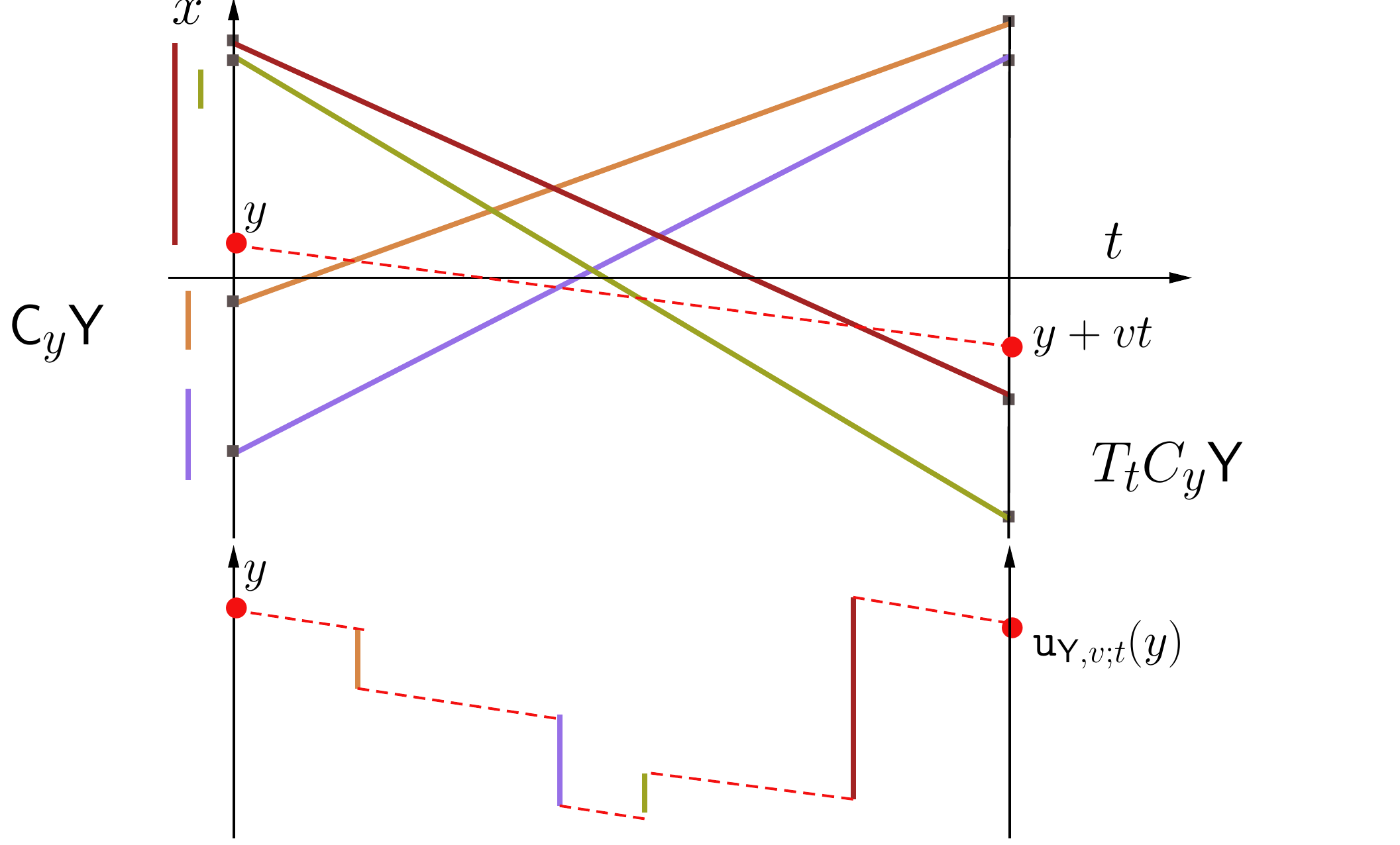}
  \caption{The upper picture shows an ideal gas evolution and the ideal trajectory of a particle $(y,v,0)$. To the left we see the length associated with each trajectory. The lower picture shows the trajectory $\ttu_{v,t}(y)[\sfY]$ associated with this configuration, and the quasiparticle trajectories.}
  \label{tagged-rod}
\end{figure}
The quasiparticle travels deterministically at speed $v$ between collisions and jumps by $\err$ when it is crossed by a slower size $\err$ quasiparticle, and by $-\err$, if the crossing quasiparticle is faster. The flow $\ttj_{C_y\sfY}$ is finite for the ideal gas configuration $C_y\sfY\in\kX$, see \eqref{kX}. The hard rod configuration at time $t$ is defined by
\begin{align}
  \label{uty}
  U_t\sfY :=  \bigl\{(\ttu_{v,t}(y),v,\err): (y,v,\err)\in\sfY\bigr\}.
\end{align}
In this definition we used that $\sfY\in\kY_y$ for all $(y,v,\err)\in\sfY$. We leave to the reader the proof of the following lemma.
\begin{lemma}[\rm Group of operators]
  The family $(U_t)_{t\in\RR}$ operating on $\kY_{\ge0}$ is a group: $ U_0\sfY= \sfY$ and $U_{t+s}\sfY= U_tU_s\sfY$ for $\sfY\in\kY_{\ge0}$.
\end{lemma}
It is convenient to express $U_t$ in terms of the ideal gas dynamics.
Let  $S_a$ be the shift operator defined by
\begin{align}
  \label{shift}
  S_a\sfX  := \{(y+a,v,\err):(y,v,\err)\in \sfX \}.
\end{align}
For $\sfY\in\kY_0$, denote the position of a zero length zero speed quasiparticle starting at the origin, by
\begin{align}
\tto_t[\sfY] &:= \ttu_{0,t}(0) = \ttj(0,0,t)[{C_0\sfY}]
               ,\qquad  \sfY\in\kY_0.  \label{hh13}
\end{align}

\begin{lemma}
Assume $\sfY\in\kY_0$. Then,
  \begin{align}
  U_t \sfY & = S_{\tto_t} D_0T_t C_0 \sfY ,\qquad \sfY \in \kY_0. \label{hh12}
\end{align}
\end{lemma}
\begin{proof}[Sketch proof] This lemma is proved in \cite{bds}; we give an idea of the proof. 
If there is no rod crossing the origin in the interval $[0,t]$, $\tto_t=0$ and for each $(x,v,\err)\in\sfX$, the mass of the configuration $T_t\sfX$ between the origin and $x+vt$ is the mass of $\sfX$ in $(0,x)$ plus the net flow $\ttj(x,v,t)$, implying \eqref{hh12} holds. On the other hand, when a particle $(\tx,\tv,\tr)\in\sfX$ cross the origin in that time interval, it shifts the configuration $D_0T_t\sfX$ by $\tr\sign(\tv)$, while $U_t\sfY$ is not shifted by these crossings. The operator $S_{\tto_t}$ in \eqref{hh12} compensates these effects.
\end{proof}
The condition $\sfY\in\kY_0$ could be dropped at the price of a more involved definition of $\tto_t$. We will take $\sfY= D_0\sfX$ which automatically belongs to $\kY_0$, avoiding that problem.

\paragraph{Remark on negative lengths} When $\sfX\in \kX\setminus\kX_{\ge0}$ we can still define $D_a\sfX$ by using \eqref{D0X}. Give an ordered label to the particles in $\sfX$, as we did in the Introduction: let $\sfX=\{(x_i,v_i,\err_i): i\in\ZZ\}$, such that $x_i<x_{i+1}$ for all $i$. The dilated configuration $\sfY=D_0\sfX =\{(y_i,v_i,\err_i):i\in\ZZ\}$, where $y_i=D_0(x_i)$ satisfies the condition
\begin{align}
  y_i+\err_i\le y_{i+1},\qquad i\in\ZZ,
\end{align}
even when $\err_i<0$.
 Assuming $y_0+\err_0\le 0<y_1$, define the compression map $C_0[\sfY]$ by
\begin{align}
  C_a(y_i) = y_i - \sum_{(y_j,v,\err)\in\sfY} \err_j\,\one\{0<j<i\}
\end{align}

\paragraph{Hard rods and Chentsov Lantuéjoul fields}
For $\sfX\in\kX$, the mass, dilation and flow in terms of field differences are: 
\begin{align}
 m_a^b(\sfX) &= \sfH(0,b)-\sfH(0,a),\quad a,b\in\RR,\label{mass1}\\[2mm]
   D_{\sfX,a}(x) &= x+\sfH(0,x)-\sfH(0,a),\label{dilation1} \\[2mm]
   \ttj_\sfX(x,v,t) &= \sfH(t,x+vt)-\sfH(0,x)\label{josefa} ,
\end{align}
where  $\sfH=\sfH[\sfX]$ is defined in \eqref{xiab}.
 The position at time $t$ of the quasiparticle associated to $(x,v,\err)$ is given by
 \begin{align}
   \tty_{v,t}(x)[\sfX] &:=  \ttu_{v,t}\bigl(D_{\sfX,0}(x)\bigr)[D_0\sfX]\notag \\
   &= x +\sfH(0,x) +vt+  \sfH(t,x+vt)-\sfH(0,x) \label{josefa90}\\[1mm]
                                                         &=x +vt+  \sfH(t,x+vt),
                                                           \label{josefa9}
 \end{align}
 in particular $\tty_{0,t}(0) = \sfH(t,0)=\tto_t$. 
If $\sfY=D_0\sfX$, \eqref{uty} and \eqref{josefa9} give
 \begin{align}
   U_t\sfY = \bigl\{\bigl(\tty_{v,t}(x),v,\err\bigr): (x,v,\err)\in \sfX \bigr\}.
   \label{josefa2}
 \end{align}
Assume $\mu\in\euM$, let $\sfX$ be a Poisson process with intensity measure $\mu$ and $\sfH=\sfH[\sfX]$. Recall the macroscopic field $H[\mu]$ defined in \eqref{hmu} and define the macroscopic dilation $\D_0[\mu]$ and flow $j[\mu]$, by
 \begin{align}
   \D_{0}(x) &:= x+ \EE \sfH(0,x) = x +H(0,x), \\
   j(x,v,t) &:= \EE \sfH(t,x+vt)-\EE \sfH(0,x) = H(t,x+vt) - H(0,x).
 \end{align}
 From \eqref{josefa90} and \eqref{josefa} we have
  \begin{align}
    y_{v,t}(x)&:=\EE\tty_{v,t}(x) = \D_{0}(x) + vt + j(x,v,t) \label{josefa3}\\
    &= x+vt +H(t,x+vt). \label{sefa}
  \end{align}

\paragraph{Law of large numbers} Let $\sfX^\eps$ be a Poisson process with intensity measure $\eps^{-1}\mu$ and  
  denote the rescaled positions
  \begin{align}
    \tty^\eps_{v,t}(x) &:= \eps \tty_{v,t}(x)[\sfX^\eps],\\
    D^\eps_{a}(x) &:= x + \eps m_a^x{\sfX^\eps}.
  \end{align}

\begin{proof}[\rm {Proof of Theorem \ref{tlln1}}]  This is a corollary to Proposition \ref{llnCLS}, by taking $\mu(dx,dv,d\err) = f(x,v,\err) \,dxdvd\err$.  Using the identities \eqref{josefa}, \eqref{josefa9}, \eqref{josefa3} and the law of large numbers \eqref{lln8} for $\sfH^\eps$, we have    
\begin{align*}
      \lim_{\eps\to0} \ttj^\eps(x,v,t) &= j(x,v,t),  \quad \PP\text{-a.s.}\\
   \lim_{\eps\to0} \tty^\eps_{v,t}(x) &= y_{v,t}(x),  \quad \PP\text{-a.s.}\\
     \lim_{\eps\to0} D^\eps_{a}(x) &= \D_{a}(x),  \quad \PP\text{-a.s.}    \tag*{\qedhere}           
\end{align*}
\end{proof}

 \paragraph{Hard rod hydrodynamics}
\label{hydrohr6}
 Let $\mu\in\euM$ and $\sfX^\eps$ be a Poisson process in $\RR^3$ with intensity measure $\eps^{-1}\mu$. Let $\sfY^\eps:=D_0\sfX^\eps$. Recall the hard rod length empirical measure at time $t$, starting with $\sfY^\eps$, 
 \begin{align}
   \ttK^\eps_t\varphi &= \eps\sum_{(y,v,\err)\in U_t\sfY^\eps}
   \err\, \varphi(y,v,\err)= \eps\sum_{(x,v,\err)\in\sfX^\eps} \err\, \varphi\bigl(\tty^\eps_{v,t}(x),v,\err\bigr), \label{ket0}
 \end{align}
by \eqref{josefa2}. 
   Denote
  \begin{align}
     \kappa_t\varphi
   := \iiint \mu(dx,dv,d\err)\,\err\,\varphi(y_{v,t}(x),v,\err).
  \end{align}
\begin{theorem}[\rm Law of Large Numbers]\label{lln}
Let the test function $\varphi$ be
  \begin{align}
    \varphi(y,v,\err)& = \phi(v,r) \,\one\{y\in[a,b]\},\quad v,\err\in\RR. \label{phi}  \end{align}
 where $\phi:\RR^2\to\RR$ is some bounded non-negative function and $a<b\in\RR$. Then,
 \begin{align}
   \lim_{\eps\to0} \ttK^\eps_t\varphi = \kappa_t\varphi,\qquad \PP\text{-a.s. and in }L_1.
   \label{josefa8}
 \end{align}
\end{theorem}
\begin{proof}
We first show the nonnegative length case, and afterwards sketch the general case. Assume $\sfX\in\kX_{\ge0}$ and denote
  \begin{align}
    \label{at1}
   \ttA^\eps_t\varphi
   &:= \eps\sum_{(x,v,\err)\in\sfX^\eps}
     r\varphi \bigl(y_{v,t}(x),v,\err \bigl).
  \end{align}
Since $\ttA^\eps_t\varphi$ is of the form $\eps\sum_{(x,v,\err)\in\sfX^\eps} \tvarphi((x,v,\err))$, Proposition \ref{llnCLS} implies
\begin{align}
  \lim_{\eps\to0} \ttA^\eps_t\varphi
  &= \kappa_t\varphi,
    \qquad \PP\text{-a.s.}  \label{josefa4}
 \end{align}
 because $\EE\ttA^\eps_t\varphi=\kappa_t\varphi$. To show that $|\ttK^\eps_t\varphi-\ttA^\eps_t\varphi|$ converges to zero, write
 \begin{align}
  \bigl|\ttK^\eps_t\varphi-\ttA^\eps_t\varphi\bigr|
   & \le \eps\sum_{(x,v,\err)\in\sfX^\eps}
     \err\,\phi(v,\err) \,
     \Bigl|\one\{\tty^\eps_{v,t}(x)\in[a,b]\bigl\}
     -\one\{y_{v,t}(x)\in[a,b]\bigl\}\Bigr|\notag\\
  &\le \eps\sum_{(x,v,\err)\in\sfX^\eps}
     \err\,\phi(v,\err) \,
   \bigl(
   \one\{x\in I[\ttx^\eps_{v,t}(a),x_{v,t}(a)]\} \notag\\
   &\hspace{35mm}+\one\{x\in I[\ttx^\eps_{v,t}(b),x_{v,t}(b)]\}\bigr), \label{a55}
 \end{align}
 where $I[z,\tz]:=[z\wedge \tz,z\vee \tz]$, and 
 \begin{align}
   \ttx^\eps_{v,t}(z)&:= \inf\{x: \tty^\eps_{v,t}(x)>z\},\label{xz1}\\
   x_{v,t}(z)&:= \inf\{x: y_{v,t}(x)>z\}. \label{xz2}
 \end{align}
 Take $a'<a< a''$. Since $\tty^\eps_{v,t}(x)\to y_{v,t}(x)$ and both $\tty^\eps_{v,t}(x)$ and $y_{v,t}(x)$ are non decreasing functions of $x$, we have $\ttx^\eps_{v,t}(z)\to x_{v,t}(z)$ for each $z\in\RR$ and consequently, for sufficiently large $\eps$, $x_{v,t}(a')\le \ttx^\eps_{v,t}(a)\le x_{v,t}(a'')$. Since the same domination holds for $x_{v,t}(a)$, we have
 \begin{align}
&\lim_{\eps\to0}     \eps\sum_{(x,v,\err)\in\sfX^\eps}
     \err\,\phi(v,\err) \,
   \one\{x\in I[\ttx^\eps_{v,t}(a), x_{v,t}(a)]\}  \label{sum90}\\
  &\qquad \le
\lim_{\eps\to0}   \eps\sum_{(x,v,\err)\in\sfX^\eps}
     \err\,\phi(v,\err) \, \one\{x\in [x_{v,t}(a'),x_{v,t}(a'')]\}\\
 &\qquad= \iiint \mu(dx,dv,d\err)\, \err\,\phi(v,\err) \, \one\{x\in [x_{v,t}(a'),x_{v,t}(a'')]\}.
 \end{align}
The same argument works for the expression involving $b$ in \eqref{a55}. Taking limits $a'\nearrow a$, $a''\searrow a$, $b'\nearrow b$ and $b''\searrow b$, and the fact that $\mu\in\euM$ and $\phi$ is uniformly bounded, we get that  \eqref{a55} goes to zero a.s.. 

To show the convergence in $L_1$, use \eqref{a55} to get 
     \begin{align}
      \EE\bigl|\ttK^\eps_0\varphi-\ttA^\eps_0\varphi\bigr|
        &\le \EE\eps\sum_{(x,v,\err)\in\sfX^\eps}
     \err\,\phi(v,\err) \,
     \Bigl( \one\{x\in I[\ttx^\eps_{v,t}(a),x_{v,t}(a)]\}\notag\\
   &\qquad\qquad+\one\{x\in I[\ttx^\eps_{v,t}(b),x_{v,t}(b)]\}\Bigr)\\
      &= \iiint \mu(dx,dv,d\err) 
     \err\,\phi(v,\err) \,
    \Bigl(\PP(x\in I[\ttx^\eps_{v,t}(a),x_{v,t}(a)]) \notag\\
   &\qquad\qquad+ \PP(x\in I[\ttx^\eps_{v,t}(b),x_{v,t}(b)])\Bigr), \label{a56}
 \end{align}
 where the identity follows from the Slyvniak-Mecke theorem \cite{MR2004226}. The expression \eqref{a56} converges to zero by bounded convergence and the fact that $\ttx^\eps_{v,t}(z)\to x_{v,t}(z)$ a.s..

To show the general case $\err\in\RR$, multiply each term of the sum \eqref{sum90} by $\one\{r\ge 0\} + \one\{r<0\}$ obtaining two sums. To conclude, apply analogous monotonicity arguments to each of the sums.
\end{proof}
\penalty -100
\paragraph{Positional fluctuations}
\begin{proof}[Proof of Theorem \ref{crf}]
 In view of \eqref{josefa9}, we have 
   \begin{align}
     \frac1{\sqrt\eps} \Bigl(\tty^\eps_{v,t}(x) - y_{v,t}(x)\Bigr)
     & \,=\, \frac1{\sqrt\eps}\;\Bigl( \sfH^\eps(t,x+vt)-\EE\sfH^\eps(t,x+vt)\Bigr),
   \end{align}
   which converges to the Levy Chentsov field, by Proposition \ref{sef1}.
   Observe that the covariances are constant in $\eps$ and given by
    \begin{align}
      \Cov\bigl(\tty^\eps_{v,t}(x),\tty^\eps_{v,t}(\tx)\bigr) &= \frac1{\eps}\EE \Bigl(\bigl(\tty^\eps_{v,t}(x) - y_{v,t}(x)\bigr)\,\bigl(\tty^\eps_{\tte,\tv}(\tx) - y_{\tte,\tv}(\tx)\bigr)\Bigr)\\
                                                                   &= \frac12\Bigl(\mu_2(\boa)+\mu_2(\bob)-\mu_2(\bab)\Bigr),
  \end{align}
  where $\ba:= (t,x+vt)$, $\bb:=(\tte,\tx+\tv\tte)$.
\end{proof}

\section{Macroscopic dynamics}
\label{S5}

\paragraph{Admissible functions}

In order to find partial differential equations and stochastic differential equations for the hydrodynamic limit solutions, in this section we assume that the measure $\mu\in \euM$ is absolutely continuous with density $f\in\euF$, defined as follows.

Measurability of all functions is assumed.
Define the set of dominated macroscopic densities as follows. 
\begin{align}
  \euL &\coloneqq \Bigl\{\gamma:\RR^2 \to \RR_{\ge0}\;:\; 
  \iint (v^2+\err^2+1)\gamma(v,\err)dvd\err<\infty\Bigr\},\\
  \euF_\gamma&\coloneqq\Bigl\{f:\RR^3\to \RR_{\geq0}: f(\cdot,v,\err)\in \C^1(\RR), \text{ and }\\
       &\hskip 1cm\max\{\|f(\cdot,v,\err)\|_\infty,\, \|\partial_x f(\cdot,v,\err)\|_\infty\}\leq \gamma(v,\err),\; (v,\err)\in \RR^2\Bigr\}.\notag
\end{align}
Define the mass and momentum functions of $f\in\cup_{\gamma\in\euL}\euF_\gamma$ by
\begin{align}
  \sigma_f(x)&:= \iint\,\err\,f(x,v,\err)\,dvd\err,\qquad
\zeta_f(x):= \iint\,v\,\err\,f(x,v,\err)\,dvd\err,\label{zet1}
\end{align}
respectively. By the Dominated Convergence Theorem, $\sigma_f$ and $\zeta_f$ are in $\C^1(\RR)$. Moreover, if $f$ is in $\euF_\gamma$, then,  $\normi{\sigma_f}$ and $\normi{\dot{\sigma}_f}$ are bounded above by $\norm{\gamma}_1$ and, $\normi{\zeta_f}$ and $\normi{\dot{\zeta}_f}$ are bounded above by $\norm{v\err\gamma(v,\err)}_1$. Define
\begin{align}
  \label{sigma>0}
  \euF &\coloneqq \Bigl\{f\in\bigcup_{\gamma\in\euL}\euF_\gamma:\sigma_f(x)>0, \text{ for all } x\in \RR\Bigr\}.
\end{align}
The condition $\sigma_f>0$ is necessary to have invertible dilation operators.

\paragraph{Dilation and contraction} Define the subset of dilated macroscopic densities as 
\[\euG := \bigl\{f\in\euF : \normi{\sigma_f}<1\bigr\}.\]
 Let $f$ in $ \euF$, $g$ {in} $\euG$ and $a $ in $\RR$.
 Define the  dilation and contraction functions $\D_{f,a}:\RR\to\RR$, $ \C_{g,a}:\RR\to\RR$, by
\begin{align}
  \D_{f,a}(b) &:= b+  \int_a^b \sigma_f(x) dx,\qquad 
 \C_{g,a}(b):= b-  \int_a^b \sigma_g(y) dy,\qquad a,b\in\RR. \notag
\end{align}
Define the dilation and contraction operators $\D_a: \euF\to\euG$, $\C_a: \euG\to\euF$ by 
\begin{align}
  \label{daf1}
  \D_af(y,v,\err) &= \frac{f(\D_{f,a}^{-1}(y),v,\err)}{1+\sigma_f(\D_{f,a}^{-1}(y))}, \quad 
  \C_ag(x,v,\err) = \frac{g(\C_{g,a}^{-1}(x),v,\err)}{1-\sigma_g(\C_{g,a}^{-1}(x))}. \end{align}
Notice that $\D_a^{-1}=\C_a$ and that the derivatives are
\begin{align}
  \label{daf2}
  \frac{d}{dx}\D_{f,a}(x) &=1+\sigma_f(x),
  \qquad
  \frac{d}{dy}\C_{g,a}(y) = 1-\sigma_g(y).
\end{align}
In particular, $\frac{d}{dy}\C_{g,a}(y)\geq 1 - \normi{\sigma_g}>0$. Thus, both functions are diffeomorphisms, the former is a dilation and the latter is a contraction.

The dilation and contraction operators conserve mass. For $f$ in $\euF$, $g$ in $\euG$ and $a$, $b$, $c$ in $\RR$, we have
\begin{align}
  \int_{\D_{f,a}(b)}^{\D_{f,a}(c)}\sigma_{\D_af}(y)dy &= \int_{b}^{c}\sigma_f(x)dx,\qquad    \int_{\C_{g,a}(b)}^{\C_{g,a}(c)}\sigma_{\C_ag}(x)dx = \int_{b}^{c}\sigma_g(y)dy.\notag
\end{align}
{This follows considering the change of variable $x = \D^{-1}_{f,a}(y) $ for the first identity and $y=\C^{-1}_{g,a}(x)$ for the second one.}
 Along the way, one can verify the following identities
\[\sigma_{\D_{a}f}\bigl(\D_{f,a}(x)\bigr)=\frac{\sigma_f(x)}{1+\sigma_f(x)}\]
\[\sigma_{\C_{a}g}\bigl(\C_{g,a}(y)\bigr)=\frac{\sigma_g(y)}{1-\sigma_g(y)}\]
The \emph{shift} operator $S_a$ applied to $f\in\euF$ or $g\in\euG$ is given by
\begin{align}
  \label{eq:shift5}
  S_af(x,v,\err) := f(x-a,v,\err).
\end{align}

\begin{proposition} Let $a$, $\qo$ and $\po$ in $\RR$.
  
  \noindent 1. If $f$ in $\euF$ and $g\coloneqq \D_{\qo}f$; then $\C_{g,\qo}=\D_{f,\qo}^{-1}$ {and $g$ in $\euG $.}

  \noindent 1'. If 
$g$ in $\euG$ and $f\coloneqq \C_{\qo}g$; then  $\C_{g,\qo}=\D_{f,\qo}^{-1}$. 

\noindent 2.  The operators $\D_{\qo}$ and $\C_{\qo}$ are inverses of each other.

 \noindent 3.  With $b\coloneqq \po - \D_{f,\qo}^{-1}(\po-a)$, the following diagram commutes.
\begin{center}
\begin{tikzcd}
& \euF \arrow[r, "\D_{\qo}"] \arrow[d, "S_b"'] & \euG \arrow[d,"S_a"]\\
& \euF \arrow[r, "\D_{\po}"'] & \euG
\end{tikzcd}
\end{center}
\noindent 4. With $c\coloneqq \C_{S_ag,\qo}(\po+a)-\po$, the following diagram commutes. (Since $\C_{\po}$ is a bijection with inverse $\D_{\po}$, we can express $c$ as function on $\euF$, by the formula $c(f) = \C_{S_a\D_{\po}f,\qo}(\po+a)-\po$.)
\begin{center}
\begin{tikzcd}
& \euG \arrow[r, "\C_{\po}"] \arrow[d, "S_a"'] & \euF \arrow[d,"S_c"]\\
& \euG \arrow[r, "\C_{\qo}"'] & \euF
\end{tikzcd}
\end{center}
\end{proposition}
\begin{proof} 
1. By the definitions, $\C_{g,\qo}(\D_{f,\qo}(\qo))=\C_{g,\qo}(\qo)=\qo$. On the other hand, by the chain rule and the definitions,
\begin{align*}
    \frac{d}{dz}\bigl(\C_{g,\qo}\circ \D_{f,\qo}(z)\bigr)\big{|}_{z=x} &=\Bigl(\bigl(\frac{d}{dz}\C_{g,\qo}(z)\bigr)\big{|}_{z=\D_{f,\qo}(x)}\Bigr)\Bigl( \bigl(\frac{d}{dz}\D_{f,\qo}(z)\bigr)\big{|}_{z=x} \Bigr)\\ &= \bigl(1-\sigma_g(\D_{f,\qo}(x))\bigr)\bigl(1+\sigma_f(x)\bigr) \\ &= \Bigl(1-\frac{\sigma_f(x)}{1+\sigma_f(x)}\Bigr)\bigl(1+\sigma_f(x)\bigr)= 1
\end{align*}
To show that $g$ in  $\euG$ we compute its mass function:
 \begin{align*}
\sigma_g(y) & = \sigma_{ \D_{\qo}f}(y) = \iint\,\err\, D_{\qo}f(y,v,\err)\,dvd\err \\ &  =   \iint\,\err\
 \frac{f(\D_{f,\qo}^{-1}(y),v,\err)}{1+\sigma_f(\D_{f,\qo}^{-1}(y))} \,dvd\err\ 
 = \dfrac{\sigma_f(\D_{f,\qo}^{-1}(y))}{1+\sigma_f(\D_{f,\qo}^{-1}(y))},
\end{align*}
and so $\normi{\sigma_g} < 1$ follows from the fact than for every $y$, $\sigma_g(y)<1 $.\\

\noindent 1'. The proof is analogous to the one before.\\

\noindent 2. Let us prove that $\C_{\qo}\circ \D_{\qo}$ is the identity operator in $\euF$. Let $f$ in $\euF$, and let $g\coloneqq \D_{\qo}f$. Then, by the first statement,  for every $(x,v,\err)$ in $\RR^3$ we have that
{
\begin{align*} 
\C_{\qo}\circ \D_{\qo} f(x,v,\err) & = \C_{\qo} g(x,v, \err) =  \frac{g(\C_{g,\qo}^{-1}(x),v,\err)}{1-\sigma_g(\C_{g,\qo}^{-1}(x))} =  \frac{g(\D_{f,\qo}(x),v,\err)}{1-\sigma_g(\D_{f,\qo}(x))}  \\ &=  \dfrac{f(x,v, \err)}{\big( 1+ \sigma_f(x) \big) \big( 1-\sigma_g(\D_{f,\qo}(x)) \big)} = f(x,v, \err)
\end{align*}
}
as claimed. 
Analogously it can be shown that $\D_{\qo}\circ \C_{\qo}$ is the identity in $\euG$.\\

\noindent 3. We want to show that {$ S_a \D_{\qo}f(x,v,\err) = \D_{\po}S_bf(x,v,\err)$, $(x,v,\err)$ in $\RR^3$. }
By definition, this is equivalent to
\begin{align} \frac{f(\D_{f,\qo}^{-1}(y-a),v,\err)}{1+\sigma_f(\D_{f,\qo}^{-1}(y-a))}=\frac{f(\D_{f,\po}^{-1}(y)-b,v,\err)}{1+\sigma_f(\D_{f,\po}^{-1}(y)-b)}
\end{align}
So it suffices to show $\D_{f,\qo}^{-1}(y-a) = \D_{f,\po}^{-1}(y)-b=:q^{(b)}$. But $q^{(b)}= \D_{f,\po-b}^{-1}(y-b)$, by definition. Hence, it suffices to prove that $y-a = q^{(b)} + \int_{\qo}^{q^{(b)}}\sigma_f(x)dx$. The identity holds for $y=\po$, by definition of $b$. Taking the derivative w.r.t. $q^{(b)}$, we obtain $1+\sigma_f(q^{(b)})$ on both sides.\\

\noindent 4. It reduces, analogously to the proof of the third statement, to verify the equality $\C_{g,\po}^{-1}(y-c)=\C_{S_ag,\qo}^{-1}(y)-a$, which certainly holds for $y=\po+c$. \qedhere
\end{proof}
\begin{obs}\label{remarkk} In particular, we have $\C_{x+z}S_{z}\D_x = S_{z}$  and $\D_{x+z}S_{z} = S_{z} \D_x$, 
for all $x$ and $z$ in $\RR$.
\end{obs}
\paragraph{Time evolution}

Define the \emph{ideal gas evolution}  operator $\T_t:\euF\to\euF$ by
\begin{align}\label{ttf22}
  \T_tf(x,v,\err):=f\circ T_{-t}(x,v,\err)=f(x-vt,v,\err).
\end{align}
This operator is a bijection with inverse $\T_{-t}$. Notice that $\T_t$ conserves $\euF$: $ f\in \euF$  if and only if $\T_tf \in \euF$.

For $f\in\euF$ define the (macroscopic) \emph{mass flow}
\begin{align}
  j_{f}^+(x,v,t)
  &:= \int_{-\infty}^{v}\int_{x}^{x+(v-w)t} \int_{-\infty}^\infty \err\,f(z,w,\err)\,d\err\,dz\,dw \label{maflow+}\\
  j_{f}^-(x,v,t)
  &:=  \int_{v}^{\infty}\int_{x+(v-w)t}^{x}\int_{-\infty}^\infty\err\,f(z,w,\err)\,d\err\,dz\,dw, \label{maflow-}\\[2mm]
  j_{f}(x,v,t) &:= j_{f}^+(x,v,t) -  j_{f}^-(x,v,t). \label{maflow}
\end{align}
The \emph{flow} $j_{f}^+(x,v,t)$ gives the mass crossing the trajectory $(s,x+vs)_{s\in[0,t]}$ with speeds less than $v$, when the initial density is $f$, and $j_{f}^-(x,v,t)$ collects the faster mass crossing the same trajectory. The net flow $j_f$ is the difference of those.
If $f\in \euF_\gamma$, for some $\gamma\in\euL$, then
\begin{align}
  \label{j+<infty}
| j_{f}^+(x,v,t)|
  &\le\int_{-\infty}^{v}\int_{-\infty}^\infty  |v-w|\,|t|\, |\err|\,\gamma(w,\err)\,d\err\,dw \,<\,\infty,
\end{align}
by definition of $\euL$. Analogously $|j_{f}^-(x,v,t)|<\infty$. We conclude that 
\begin{align}
  \label{j<infty}
  |j_{f}^\pm(x,v,t)| <\infty,\;\; |j_{f}(x,v,t)|<\infty,\quad f\in\euF,\; x,v,t\in \RR.
\end{align}
Define the \emph{hard-rod evolution operator} $\U_t:\euG\to\euG$ by
\begin{align}
   \U_tg(y,v,\err) &:=  S_{\tto_t}\D_0\T_t\C_{0} g(y,v,\err) \label{hre1}
\\
  \tto_t&:= j_{\C_0g}(0,0,t).\label{otg2}
 \end{align}
 As in the microscopic case, $\D_0\T_t\C_0 g$ is the macroscopic hard rod evolution as seen from a zero-speed, zero-length quasiparticle $(0,0,0)$, added at the origin at time $0$, and $\tto_t$ is the position of this quasiparticle at time $t$. The shift by $\tto_t$ is performed in order to obtain the hard rod evolution of $g$ (as seen from the origin).

For every $g$ in $\euG$, we define the macroscopic position at time $t$ of a hard rod $(q,v,\err)$, as the bijection $u_{g,v,t}:\RR\to \RR$
 \begin{align}
   \label{ugt}
   u_{g,v,t}(q) := q+ vt +  j_{\C_qg}(q,v,t).
 \end{align}
Notice that this does not depend on the length $\err$.

\begin{lemma}[\rm Quasiparticle evolution]\label{uderivativeu} Let $g$ in $\euG$ and $(q,v,t)$ in $\RR^3$. Then $u_{g,v,t}:\RR\to\RR$ is a diffeomorphism and, for every $(p, \vo )$ in $\RR^2$, the following formulas hold.
  \begin{align}
    u_{g,v,t}(q)
    &= \C_{g,p}(q) + vt + j_{\C_p g}(p,\vo ,t) + \int_{p + \vo t}^{\C_{g,p}(q) + vt}\sigma_{\T_t \C_{p}g}(x)dx\\
    \frac{d}{dq}u_{g,v,t}(q)
    &= \bigl(1-\sigma_g(q)\bigr)\bigl(1+\sigma_{\T_t\C_{p}g}(\C_{g,p}(q)+vt)\bigr)
  \end{align}
\end{lemma}
\begin{proof} Since $\C_qg(x,v,\err) = \C_{a}g(x-q + \C_{g,a}(q),v,\err)$, a change of variables give us the following alternative formulation. Note that the function we integrate is different and does not depend on $q$.
\begin{align}
   j_{\C_qg}(q,v,t) = \int\err\int\int_{\C_{g,a}(q)}^{\C_{g,a}(q)+(v-\tw )t}\C_{a}g(x,\tw ,\err)dxd\tw d\err. 
\end{align}
Then,
\begin{align*}
  &\hspace{-3mm}u_{g,v,t}(q)
  \coloneqq q + vt + j_{\C_qg}(q,v,t) \\
  &\hspace{-1mm}= q + vt + \int\err\int\int_{\C_{g,p}(q)}^{\C_{g,p}(q)+(v-\tw )t}\C_{p}g(x,\tw ,\err)dxd\tw d\err \\
  &\hspace{-1mm}= q + vt + \int\err\int\Bigl(\int_{\C_{g,p}(q)}^{p}\C_{p}g(x,\tw ,\err)dx + \int_{p}^{p + (\vo -\tw )t}\C_{p}g(x,\tw ,\err)dx  
  \\ 
  &\hspace{47mm}+ \int_{p + (\vo -\tw )t}^{\C_{g,p}(q)+(v-\tw )t}\C_{p}g(x,\tw ,\err)dx\Bigr)d\tw d\err.
\end{align*}
Now, the first integral, together with the term $q$, is equal to $\C_{g,p}(q)$. The second integral is equal to $j_{\C_pg}(p,\vo ,t)$. Finally, by the change of variable $y \coloneqq x + \tw t$, we can see that the third integral is equal to $\int_{p + \vo t}^{\C_{g,p}(q) + vt}\sigma_{\T_t\C_{p}g}(x)dx$. The formula for $u_{g,v,t}$ follows, from which we can deduce the formula for its derivative as it is done below.

Denote $\hat{q}\coloneqq \C_{g,p}(q)$ and compute 
\begin{align*}
    \frac{d}{dq}u_{g,v,t}(q) &= \frac{d\hat{q}}{dq}(q) + \frac{d}{dq}\Bigl(\int_{p + \vo t}^{\hat{q} + vt}\sigma_{\T_t\C_{p}g}(x)dx\Bigr) \\ &= \frac{d\hat{q}}{dq}(q) + \frac{d}{d\hat{q}}\Bigl(\int_{p + \vo t}^{\hat{q} + vt}\sigma_{\T_t\C_{p}g}(x)dx\Bigr) \frac{d\hat{q}}{dq}(q) \\ &= \frac{d\hat{q}}{dq}(q)\Bigl(1 + \frac{d}{d\hat{q}}\Bigl(\int_{p + \vo t}^{\hat{q} + vt}\sigma_{\T_t\C_{p}g}(x)dx\Bigr)\Bigr) \\ &= \bigl(1-\sigma_g(q)\bigr)\bigl(1 + \sigma_{\T_t \C_{p}g}(\hat{q} + vt)\bigr).\tag*{\qedhere}
\end{align*}
\end{proof}

\begin{corollary}[\rm Monotonicity and smoothness] For every $g$ in $\euG$ and every $(v,t)$ in $\RR^2$, the function $u_{g,v,t}$ is increasing and smooth.
\end{corollary}

We still need some previous lemmas before facing the proof of the existence and uniqueness theorem for the hydrodynamic equation. The following one is the macroscopic counterpart of \eqref{hh12}. 

\begin{lemma}
  [Evolution formula]\label{lemevol} Let $g$ in $\euG$. Then, for $p,q,v,\vo,t\in\RR$,  
\begin{align}
  g (u_{g,v,t}^{-1}(q),v,\err)\frac{d}{dq}u_{g,v,t}^{-1}(q)
  &= S_a \D_{p+\vo t}\T_t\C_{p}g(q,v,\err),\label{difgt}
\end{align}
where $a = j_{\C_pg}(p,\vo,t)$.
\end{lemma}
\begin{proof} Let us write $f\coloneqq \C_{p}g$, $\hat{q}\coloneqq \C_{g,p}\bigl(u_{g,v,t}^{-1}(q)\bigr)$ and $u=u_{g,v,t}$.
By Lemma \ref{uderivativeu}, we have that
\begin{align}\label{trefoil}
\begin{split}
    \frac{d}{dq}u^{-1}(q) &= \Bigl(\frac{d}{dx}u(x)\big{|}_{x=u^{-1}(q)}\Bigr)^{-1}  \\ &
    = \Bigl(1-\sigma_{g}\bigl(u^{-1}(q)\bigr)\Bigr)^{-1}\Bigl(1+\sigma_{\T_tf}(\hat{q}+vt)\Bigr)^{-1}.
\end{split}
\end{align}
We claim that
\begin{align}\label{diamond}
    \D_{\T_tf,p+\vo t}(\hat{q}+vt) = q-a.
\end{align}
Indeed, first we see that, by the involved definitions,
\begin{align*}
    j_f(p,v,t)-j_f(p,\vo,t) &= \int_0^{+\infty}\err\int_{-\infty}^{+\infty}\int_{p+(\vo-\tw )t}^{p+(v-\tw )t}f(x,\tw ,\err) dxd\tw d\err \\ &= \int_{p+\vo t}^{p+vt}\sigma_{\T_tf}(x)dx = \D_{\T_tf,p+\vo t}(p+vt) - (p + vt).
\end{align*}
From the last equation, together with Lemma \ref{uderivativeu}, we see that \eqref{diamond} holds for $q = u(p)$. Second, taking the derivative w.r.t.~$q$ on \eqref{diamond}, applying the chain rule and using \eqref{trefoil}, we get
\begin{gather*}
  \frac{d}{dq}\D_{\T_tf,p+\vo t}(\hat{q}+vt) 
  = \Bigl(\frac{d}{dq}\D_{\T_tf,p+\vo t}(q)\Bigr)\Big{|}_{q=\hat{q}+vt}\frac{d\hat{q}}{dq}(q)
  \\ 
  = \Bigl(1+\sigma_{\T_tf}(\hat{q}+vt)\Bigr)\Bigl(\frac{d}{dq}\C_{g,p}(q)\Bigr)\Big{|}_{q=u_{g,v,t}^{-1}(q)}\frac{d}{dq}u_{g,v,t}^{-1}(q)
  \\ 
  = \Bigl(1+\sigma_{\T_tf}(\hat{q}+vt)\Bigr)\Bigl(1-\sigma_{g}\bigl(u_{g,v,t}^{-1}(q)\bigr)\Bigr)\frac{d}{dq}u_{g,v,t}^{-1}(q)
  = 1= \frac{d}{dq}(q-a).
\end{gather*}
Thus equation \eqref{diamond} holds, as claimed. 
Finally, by definition of $\hat{q}$, we have that $u^{-1}(q) = \C_{g,p}^{-1}(\hat{q})$. Using this identity, together with equations \eqref{trefoil} and \eqref{diamond}, we get
\begin{align*}
  g (u^{-1}(q),v,\err)\frac{d}{dq}u^{-1}(q) &= \frac{g\bigl(\C_{g,p}^{-1}(\hat{q}),v,\err)\bigr)}{\bigl(1-\sigma_g(\C_{g,p}^{-1}(\hat{q}))\bigr)\bigl(1+\sigma_{\T_tf}(\hat{q}+vt)\bigr)}
  \\ &= \frac{f(\hat{q},v,\err)}{(1 + \sigma_{\T_tf}(\hat{q}+vt))}
  \\ &= \frac{\T_tf(\hat{q} + vt,v,\err)}{(1 + \sigma_{\T_tf}(\hat{q}+vt))}
  \\ &= \D_{p+\vo t} \T_tf (\D_{\T_tf,p+\vo t}(\hat{q}+vt), v, \err)
  \\ &= \D_{p+\vo t} \T_tf (q-a, v, \err)
  \\ &= S_a \D_{p+\vo t} \T_tf (q, v, \err)
  \\ &= S_a \D_{p+\vo t} \T_t \C_{p}g (q, v, \err).\tag*{\qedhere}
\end{align*}
\end{proof}
\begin{corollary}[\rm Alternative evolution formula] We have
	\begin{align}
	\U_tg (q,v,\err) = g (u_{g,v,t}^{-1}(q),v,\err)\frac{d}{dq}u_{g,v,t}^{-1}(q).
	\end{align}
\end{corollary}
 \begin{lemma}[\rm Group property]\label{grouproperty}
 The family $(\U_t)_{t\in\RR}$ is a group: $U_0=$ Identity and
   $ \U_{t+s}= \U_t\U_s$, for $t,s\in\RR$.
 
%
\end{lemma}
\begin{proof} It is clear that $\U_0 g = g$. 
Let $t$ and $s$ in $\RR$. We want to prove that $\U_{s}\U_{t}g = \U_{t+s}g$. Let us choose $(p, \vo)$ in $\RR^2$ arbitrarily. Let $a \coloneqq j_{\C_p {g}}(p,\vo,t)$, $b \coloneqq j_{{\C_{p + \vo t + a}}\U_{t}g}(p + \vo t + a, \vo, s)$ and $c \coloneqq j_{ {\C_p} g}(p,\vo,t+s)$. From their definitions, it can be proved that $c = a + b$. Now, we compute
\begin{align*}
  \U_{s}\U_{t}g &= \bigl(S_{b}\D_{p + \vo(t+s)+a}\T_{s}\C_{p + \vo {t}  + a}\bigr)\bigl(\U_{t}g\bigr)
                      \tag*{\text{(by Lemma \ref{lemevol})}}
  \\ &= \bigl(S_{b}\D_{p + \vo(t+s)+a}\T_{s}\C_{p + \vo {t} + a}\bigr)
       \bigl(S_{a}\D_{p + \vo t}\T_{t}\C_{p}g\bigr)\tag*{\text{(by Lemma \ref{lemevol})}}
    \\ &= S_{b}\D_{p + \vo(t+s)+a}\T_{s}\bigl(\C_{p + \vo {t}  + a}S_{a}\D_{p + \vo t}\bigr)\T_{t}\C_{p}g
  \\ &= S_{b}\D_{p + \vo(t+s)+a}\bigl(\T_{s}S_{a}\bigr)\T_{t}\C_{p}g
       \tag*{\text{( {by Remark \ref{remarkk}})}}
  \\ &= S_{b}\bigl(\D_{p + \vo(t+s)+a}S_{a}\bigr)\bigl(\T_{s}\T_{t}\bigr)\C_{p}g
       \tag*{\text{($S$ and $T$ commute)}}
  \\ &= \bigl(S_{b}S_{a}\bigr)\D_{p + \vo(t+s)}\T_{t+s}\C_{p}g
       \tag*{\text{(  {by Remark \ref{remarkk}})}}
    \\ &= \bigl(S_{a + b}\bigr)\D_{p + \vo(t+s)}\T_{t+s}\C_{p}g 
    \\ &= \bigl(S_{a+b}\D_{p + \vo(t+s)}\T_{t+s}\C_{p}g\bigr)
    \\ &= \U_{t+s}g. \tag*{\text{(by Lemma \ref{lemevol})}\quad\qedhere}
\end{align*}
\end{proof}

The next lemma shows the equivalence between the definitions \eqref{gtf9} and \eqref{hre1} of $g_t$.
\begin{lemma}[\rm Equivalence]
  \label{lmuf1}
    Let $f\in\euF$, $g:=\D_0f$, $g_t:= \cS_{o_t}\D_0\T_t\C_0g$ and $y_{f,v,t}(x):= \D_{f,0}(x)+vt+j_f(x,v,t)$. Then,
      \begin{align}
       &\iiint dxdvd\err\, f(x,v,\err)\,\err\, \varphi(y_{f,v,t}(x),v,\err)\\
        &\qquad=\iiint dqdvd\err\, g_t(q,v,\err)\,\err\, \varphi(q,v,\err). \label{mug}
      \end{align}
    \end{lemma}
    \begin{proof}
Lighten notation by writing $y_{v,t}= y_{f,v,t}$ and $u_{v,t}=u_{g,v,t}$, and recall $y_{v,t}(x)= u_{v,t}(D_0(x))$. Changing variables $x=y^{-1}_{v,t}(q)$, we have
\begin{gather*}
  dx = \frac{d}{dq} y^{-1}_{v,t}(q)\, {dq} = \frac{d}{dq} \C_{g,0}(u^{-1}_{v,t}(q))\, dq
  =
    (1-\sigma_g(u^{-1}_{v,t}(q)) \,\frac{d}{dq} u^{-1}_{v,t}(q)\, dq,
\end{gather*} 
by \eqref{daf2}, and
    \begin{align*}
      &\hspace{-2mm} \iiint dxdvd\err\, f(x,v,\err)\,\err\, \varphi(y_{v,t}(x),v,\err)\\
        &=\iiint d{q}dvd\err\, f\bigl(y^{-1}_{v,t}({q}),v,\err\bigr)\,\err\, \varphi({q},v,\err)\,
          (1-\sigma_g(u^{-1}_{v,t}({q})) \,\frac{d}{d{q}} u^{-1}_{v,t}({q})\\
        &=\iiint d{q}dvd\err\, f\bigl(\C_{g,0}u^{-1}_{v,t}({q}),v,\err\bigr)\,\err\, \varphi({q},v,\err)\,
          (1-\sigma_g(u^{-1}_{v,t}({q})) \,\frac{d}{d{q}} u^{-1}_{v,t}({q})\\
        &=\iiint d{q}dvd\err\, g\bigl(u^{-1}_{v,t}({q}),v,\err\bigr)\,\err\, \varphi({q},v,\err)\,\frac{d}{d{q}} u^{-1}_{v,t}({q}),\qquad \text{by \eqref{daf1}}\\
        &=\iiint d{q}dvd\err\, g_t\bigl({q},v,\err\bigr)\,\err\, \varphi({q},v,\err),\qquad \text{by \eqref{difgt}.}
    \end{align*}
    In the second identity we used that \eqref{josefa13} implies
$y^{-1}_{v,t}({q}) = \C_{g,0}u^{-1}_{v,t}({q})$.
  Indeed,  $y_{v,t}(x)=  u_{v,t}(D_{f,0}(x))$, where $u_{v,t}=u_{g,v,t}$ is defined in \eqref{ugt}. 
\end{proof}

\paragraph{Proof of the density evolution theorem}\penalty 10000
Recall the definition of $\sigma_f$ and $\zeta_f$ in \eqref{zet1} and denote
\begin{gather}
  \label{sigmaxt}
  \sigma_g(q,t)= \sigma_{g_t}(q)\quad 
  \zeta_g(q,t)= \sigma_{g_t}(q).
\end{gather}
\begin{proof}[Proof of Theorem \ref{cauchy-theorem}]
We first prove that {$g_t=\U_tg$ satisfies the evolution equation}. It is clear that $g$ satisfies the initial condition. Now, by the group property of Lemma \ref{grouproperty}, it suffices to verify the differential equation at time $t=0$. More precisely, for the general case $t=t_0$ we could satisfy the differential equation for $t=0$ and the initial condition $\U_{t_0}\tg$, using the function $\U_t\U_{s}\tg = \U_{t+s}\tg$. Then simply by shifting its time dependence we would get that $\U_tg$ satisfies the differential equation for $t=s$.
	
Let us write $u(q,v,t)\coloneqq u_{\tg,v,t}(q)$ and $\hat{u}(q,v,t)\coloneqq u_{\tg,v,t}^{-1}(q) \eqqcolon \hy$. Now, first we compute, using the chain rule,
	\begin{align*}
	1 &= \partial_q(q)
             = \partial_q\bigl(u(\hat{u}(q,v,t),v,t)\bigr)
             = \bigl(\partial_qu(\hy ,v,t)\bigr) \bigl(\partial_q\hat{u}(q,v,t)\bigr).
	\end{align*}
	Then $\partial_q\hat{u}(q,v,t) = \bigl(\partial_qu(\hy ,v,t)\bigr)^{-1}$. Next,
	\begin{align*}
	0 &= \partial_t(q)
             = \partial_t\bigl(u(\hat{u}(q,v,t),v,t)\bigr)
             = \bigl(\partial_qu(\hy ,v,t)\bigr) \bigl(\partial_t\hat{u}(q,v,t)\bigr) + \partial_tu(\hy ,v,t).
	\end{align*}
	Then $\partial_t\hat{u}(q,v,t)=-\frac{\partial_tu(\hy ,v,t)}{\partial_qu(\hy ,v,t)}$, whenever $\partial_qu(\hy ,v,t)\neq0$, which is always the case.
Thus, evaluating at $t=0$, we get
	\begin{align*}
          \partial_q \hat{u}(q,v,t)\big{|}_{t=0}
          &= 1
          \\
          \partial_t \hat{u}(q,v,t)\big{|}_{t=0}
          &= -v -\bigl(1-\sigma_{\tg}(q)\bigr)^{-1}\int\terr\int(v-w)\tg(q,w,\terr)dwd\terr
          \\
          \partial_q\partial_t \hat{u}(q,v,t)\big{|}_{t=0}
          &= -\partial_q\Bigl(\bigl(1-\sigma_{\tg}(q)\bigr)^{-1}
            \int\terr\int(v-w)\tg(q,w,\terr)dwd\terr\Bigr).
	\end{align*}
	Then, for the time derivative of the proposed solution $g$ at time $t=0$, we finally get
	\begin{align*}
          \partial_tg_t
          &(q,v,\err)\big{|}_{t=0}
            = \Bigl(\partial_t\bigl(\tg(\hy ,v,\err) \partial_q\hat{u}(q,v,t)\bigr)\Bigr)\Big{|}_{t=0} 
          \\
          &= \Bigl(\partial_q\tg(\hy,v,\err)\partial_t\hat{u}(q,v,t)\partial_q\hat{u}(q,v,t)\Bigr)\Big{|}_{t=0}
            \tg(\hy ,v,\err)\Bigl(\partial_t\partial_q\hat{u}(q,v,t)\Bigr)\Big{|}_{t=0}
          \\
          &= \partial_q\tg(\hy ,v,\err)\Bigl(-v - \bigl(1-\sigma_{\tg}(q)\bigr)^{-1}
            \int\terr\int(v-w)\tg(q,w,\terr)dwd\terr\bigr)\Bigr)
          \\
          &\quad+ \tg(\hy ,v,\err)\Bigl(-\partial_q\Bigl(\bigl(1-\sigma_{\tg}(q)\bigr)^{-1}
            \int\terr\int(v-w)\tg(q,w,\terr)dwd\terr\Bigr)\Bigr)
          \\
          &= -v\partial_q\tg(\hy ,v,\err) -\partial_q \Bigl(\tg(\hy ,v,\err)\bigl(1-\sigma_{\tg}(q)\bigr)^{-1}
          \\ &\qquad\qquad\qquad\qquad\qquad\qquad
               \times 
               \int\terr\int(v-w)\tg(q,w,\terr)dwd\terr\Bigr)
          \\ &= -v\partial_qg_t(q,v,\err)\Big{|}_{t=0} -\partial_q \Bigl(g_t(\hy ,v,\err)\bigl(1-\sigma_g(q,t)\bigr)^{-1}
          \\ &\qquad \qquad\qquad\qquad\qquad\qquad
               \times\int\terr\int(v-w)g_t(q,w,\terr)dwd\terr\Bigr)\Big{|}_{t=0},
	\end{align*}
	which is precisely the differential equation, at time $t=0$, of our Cauchy problem \eqref{cauchy-g}.
	
        Now we prove uniqueness. 
     Let us suppose that $h_t$ is another solution to the Cauchy problem.    First, let us note that (first multiplying by $\err$ and then) integrating (w.r.t. $\err$ and $v$) the differential equation of the Cauchy problem \eqref{cauchy-g}, we get the following integral version
	\begin{align}\label{conteq}
	\partial_t\sigma_h(q,t)+\partial_q\zeta_h(q,t) = 0,
	\end{align}
	which is a continuity equation, expressing the conservation of mass law, and will be used later. Secondly, let us define the evolving effective velocity field associated to $h$ as the function
 	\[v_h^{\eff}(q,v,t) := v + \frac{\int\terr\int(v-w)h(q,w,\terr,t)dwd\terr}{1-\sigma_h(q,t)} \in \RR.\]
	Note that we have the following equivalent formulation,
	\begin{align}\label{veffeq}
	v_h^{\eff}(q,v,t) = \frac{v - \zeta_h(q,t)}{1-\sigma_h(q,t)}.
	\end{align}
 which, in particular gives
\[v^{\eff}(y,v,t)-v^{\eff}(y,w,t) = \frac{v-w}{1-\sigma(y,t)}. \]

	Now we can consider the following Cauchy problem
	\begin{align}\label{auxedo}
	\begin{cases}
	&\dot{q}_{h,q,v} = v_h^{\eff}(q_{h,q,v}(t),v,t)
	\\ &q_{h,q,v}(0) = q
	\end{cases}
	\end{align}
	Since $h$ is in $\euG$, for every $v$ in $\RR$ we have that $v_h^{\eff}(\cdot,v,\cdot)$ and $\partial_qv_h^{\eff}(\cdot,v,\cdot)$ are in $\C^2(\RR^2)$. Then, by the general theory of ODEs, for each $(q,v)$ in $\RR^2$, there exists a global solution to the problem \eqref{auxedo}. Let us define
	\[
          u_{h,v,t}(q):= q_{h,q,v}(t),
        \]
	which is a diffeomorphism. The rest of the proof consists of showing, for every $(q,v,\err,t)$ in $\RR^3\times\RR_{\geq0}$, that the following two formulas hold.
	\begin{align}\label{finalstep}
	\begin{cases}
	&h(q,v,\err,t) = \tg(u_{h,v,t}^{-1}(q),v,\err)\frac{d}{dq}u_{h,v,t}^{-1}(q),
	\\ &u_{h,v,t}=u_{\tg,v,t}
	\end{cases}
	\end{align}
	which gives uniqueness. We split the proof in three parts.\\
	
	\underline{Part} 1: We claim that, for all $p,\vo,q,v$ in $\RR$, the following holds.
	\begin{align}\label{auxeq1}
	\C_{h_t,p(t)}(q(t)) - \C_{\tg,p}(q) + (\vo-v)t - p(t) + p = 0,
	\end{align}
	where $h_t \coloneqq h(\cdot,\cdot,t,\cdot)$, $p(t)\coloneqq q_{h,p,\vo}(t)$ and $q(t)\coloneqq q_{h,q,v}(t)$ for every $t$ in $\RR$. This equation clearly holds for $t=0$ and, applying the Leibniz integral rule to compute its time derivative, and taking into account the continuity equation \eqref{conteq} and identity \eqref{veffeq}, we get the following result.
\begin{align*}
&	\frac{d}{dt}\Bigl(\C_{h_t,p(t)}(q(t)) - \C_{\tg,p}(q) + (\vo-v)t - p(t) + p\Bigr) 
          \\ &= \frac{d}{dt}\Bigl(\int_{p}^q\sigma_{h}(x)dx - \int_{p(t)}^{q(t)}\sigma_h(x,t)dx
               + (\vo-v)t + q(t)- p(t) + p-q\Bigr) 
          \\ &= \dot{p}(t)\sigma_h(p(t),t) - \dot{q}(t)\sigma_h(q(t),t)
               - \int_{p(t)}^{q(t)}\partial_t\sigma_h(x,t)dx +\vo-v + \dot{q}(t)-\dot{p}(t)
  \\ &= \int_{p(t)}^{q(t)}\partial_x\zeta_h(x,t)dx + \vo-v + \dot{q}(t)\bigl(1-\sigma_h(q(t),t)\bigr)
               - \dot{p}(t)\bigl(1-\sigma_h(p(t),t)\bigr) 
  	\\[2mm] &= \zeta_h(q(t),t)-\zeta_h(p(t),t) + \vo-v 
          \\ &\qquad+\enspace v_h^{\eff}(q(t),v,t)\bigl(1-\sigma_h(q(t),t)\bigr)
               - v_h^{\eff}(p(t),\vo,t)\bigl(1-\sigma_h(p(t),t)\bigr)
	\\[2mm] &= \zeta_h(q(t),t)-\zeta_h(p(t),t) + \vo-v
	\\ &\qquad+\enspace \frac{v - \zeta_h(q(t),t)}{1-\sigma_h(q(t),t)}\bigl(1-\sigma_h(q(t),t)\bigr) - \frac{\vo - \zeta_h(p(t),t)}{1-\sigma_h(p(t),t)}\bigl(1-\sigma_h(p(t),t)\bigr)
          \\[2mm] &= \zeta_h(q(t),t)-\zeta_h(p(t),t) + \vo-v
               + v - \zeta_h(q(t),t) - \bigl(\vo - \zeta_h(p(t),t)\bigr)
	\\ &= 0,
	\end{align*}
where we used \eqref{conteq}, \eqref{auxedo} and \eqref{veffeq} to see the third, fourth and fifth identities, respectively.	Thus equation \eqref{auxeq1} holds, as claimed.\\
	
	\underline{Part 2}: We claim that, for every $q,v,\err,t$ in $\RR$, the following holds.
	\begin{align}
	&h(u_{h,v,t}(q),v,\err,t) \frac{d}{dq}u_{h,v,t}(q) = \tg(q,v,\err) ,
	\\ &\C_{q(t)}h_t= S_{a(t)}\T_t\C_q\tg,
\label{prefinalstep}
	\end{align}
	where $q(t)\coloneqq u_{h,v,t}(q) $ and $a(t)\coloneqq q(t)-q-vt$.
	
	Indeed, regarding the first equation, we see that it holds for $t = 0$ and we compute the time derivative of the LHS applying the chain rule as follows. Write $u$ instead of $u_{h,v,t}$ to lighten notation and get
  \begin{align*}
  &\partial_t\Bigl(
  h \big(u(q,v,t),v,\err,t\bigr) \partial_qu(q,v,t)\Bigr) 
  \\
  &= \partial_t\Bigl(h\bigl(u(q,v,t),v,\err,t\bigr)\Bigr)\partial_qu(q,v,t)
       + h\bigl(u(q,v,t),v,\err,t\bigr)\partial_t\partial_qu(q,v,t)\Bigr)
  \\
  &= \Bigl(\partial_th\bigl(u(q,v,t),v,\err,t\bigr)
       + \partial_qh\bigl(u(q,v,t),v,\err,t\bigr)\partial_tu(q,v,t)\Bigr)\partial_qu(q,v,t)
  \\ 
    &\qquad+\enspace h\bigl(u(q,v,t),v,\err,t\bigr)\partial_q\partial_tu(q,v,t)
    \\ &
           = \Bigl(\partial_th\bigl(q(t),v,\err,t\bigr) 
       + \partial_qh\bigl(q(t),v,\err,t\bigr)v_h^{\eff}\bigl(u(q,v,t),v,t\bigr)\Bigr)\partial_qu(q,v,t) 
  \\ &\qquad+\enspace h\bigl(q(t),v,\err,t\bigr)\partial_q\Bigl(v_h^{\eff}\bigl(u(q,v,t),v,t\bigr)\Bigr)
     \tag*{\text{(because $\dot q=v^{\eff}$)}}  \\
&  = \Bigl(\partial_th(q(t),v,\err,t)
   + \partial_qh(q(t),v,\err,t)v_h^{\eff}\bigl(q(t),v,t\bigr)\Bigr)\partial_qu(q,v,t) 
   \\ &\qquad+\enspace h(q(t),v,\err,t)\partial_qv_h^{\eff}\bigl(q(t),v,t\bigr)\partial_qu(q,v,t)
   \\ 
    &= \Bigl(\partial_th(q(t),v,\err,t)
       + \partial_q\Bigl(h(q(t),v,\err,t)v_h^{\eff}\bigl(q(t),v,t\bigr)\Bigr)\Bigr) \partial_qu(q,v,t)
       = 0,
\end{align*}
since $h$ is a solution to the Cauchy problem \eqref{cauchy-g}. Thus the first equation holds. Regarding the second equation, which is specifically the following
	\[\frac{h(\C_{h_t,q(t)}^{-1}(x),w,\err,t)}{1-\sigma_h(\C_{h_t,q(t)}^{-1}(x),t)} = \frac{\tg(\C_{\tg,q}^{-1}\bigl(x+(v-w)t + q - q(t)\bigr),w,\err)}{1-\sigma_{\tg}(\C_{\tg,q}^{-1}\bigl(x+(v-w)t + q - q(t)\bigr))},\]
	we first note that, since the composition $\C_{h_t,q(t)}\circ u_{h,v,t}:\RR\to\RR$ is a diffeomorphism, we can replace $x$ in the equation above by $\C_{h_t,q(t)}(u_{h,v,t}(q)) = \C_{h_t,q(t)}(q(t))$. After doing this change of variables, and applying the equation \eqref{auxeq1} with $q=p$, we see that the resulting equation is
	\[\frac{h(q(t),v,\err,t)}{1-\sigma_h(q(t),t)} = \frac{\tg(q,v,\err)}{1-\sigma_{\tg}(q)},\]
	which, we claim, is equivalent to the already proved first equation:
	\begin{align}\label{uhderivative}
	\frac{d}{dq}u_{h,v,t}(q)=\frac{1-\sigma_{\tg}(q)}{1-\sigma_h(q(t),t)}.
	\end{align}
	Indeed, this equation holds for $t=0$ and, before taking its time derivative, we note that
	\begin{align}\label{veffdereq}
	\partial_qv_h^{\eff}(q,v,t) &= \frac{-\partial_q\zeta_h(q,t)\bigl(1-\sigma_h(q,t)\bigr) + \partial_q\sigma_h(q,t)\bigl(v-\zeta_h(q,t)\bigr)}{\bigl(1-\sigma_h(q,t)\bigr)^2}
	\\ &= \frac{-\partial_q\zeta_h(q,t) + \partial_q\sigma_h(q,t)v_h^{\eff}(q,v,t)}{1-\sigma_h(q,t)},\label{ch144}
	\end{align} 
	which follows directly from \eqref{veffeq}. Now, the time derivative of the (conveniently rearranged) expression \eqref{uhderivative}, gives
	\begin{align*}
          \partial_t\Bigl(
          &\partial_qu(q,v,t)\bigl(1-\sigma_h(q(t),t)\bigr) - \bigl(1-\sigma_{\tg}(q)\bigr)\Bigr)
          \\ &= \partial_t\Bigl(\partial_qu(q,v,t)\Bigr)\bigl(1-\sigma_h(q(t),t)\bigr)
               + \partial_qu(q,v,t)\partial_t\Bigl(1-\sigma_h(q(t),t)\Bigr)
	\\ &= \partial_q\Bigl(v_h^{\eff}(q(t),v,t)\Bigr)\bigl(1-\sigma_h(q(t),t)\bigr)
          \\ &\qquad+\enspace \partial_qu(q,v,t)\Bigl(-\partial_t\sigma_h(q(t),t)
               -\partial_q\sigma_h(q(t),t)\partial_tu(q,v,t)\Bigr)
	\\ &= \partial_qv_h^{\eff}(q(t),v,t)\partial_qu(q,v,t)\bigl(1-\sigma_h(q(t),t)\bigr)
          \\ &\qquad+\enspace \partial_qu(q,v,t)\Bigl(\partial_q\zeta_h(q(t),t)
               -\partial_q\sigma_h(q(t),t)v_h^{\eff}(q(t),v,t)\Bigr)
             = 0,
	\end{align*}
        where the last two identities follow from \eqref{conteq} and \eqref{ch144}, respectively.
	Thus the equation \eqref{uhderivative} holds, as claimed.\\
	
	\underline{Part 3}: We prove \eqref{finalstep}. The first equation in \eqref{finalstep} follows from \eqref{prefinalstep} and using \eqref{uhderivative}. To show the second line in \eqref{finalstep}, use \eqref{prefinalstep} to get
	\begin{align*}
	v_h^{\eff}(q(t), v, t) &= v + \int\err\int(v-w)\C_{q(t)}h(q(t),w,\err,t)dwd\err
	\\ &= v + \int\err\int(v-w)\T_t\C_q\tg(q + vt,w,\err)dwd\err
	\\ &= v + \partial_tj_{\tg}(q,v,t),
	\end{align*}
	by definition of $j_{\tg}$. Then integrate in $t$ and use \eqref{auxedo} to get
          \begin{align}
            q(t) - q 
            &= \int_0^t v_h^{\eff}(q(s), v, s)ds
            = \int_0^t\bigl(v + \partial_tj_{\tg}(q,v,s)\bigr)ds\notag
            \\           
             &  = vt + j_{\tg}(q,v,t) - j_{\tg}(q,v,0)
            = vt + j_{\tg}(q,v,t)= u_{\tg,v,t}(q).\notag
          \end{align}
  We proved $u_{g,v,t}(q)=u_{h,v,t}(q)$ up to a constant.         
\end{proof}


\phantomsection
\addcontentsline{toc}{section}{References}
\bibliographystyle{plain}
\bibliography{hydrorod-em}

\newpage

\noindent Pablo A. Ferrari

\noindent \address{Departamento de Matemática, Facultad de Ciencias Exactas y Naturales, Universidad de Buenos Aires and IMAS-CONICET,\\  Buenos Aires, Argentina}

\noindent \email{pferrari@dm.uba.ar}


\bigskip

\noindent Chiara Franceschini

\noindent \address{Department of mathematics, University of Modena and Reggio Emilia,  \\ Modena, Italy}

\noindent \email{cfrances@unimore.it}

\bigskip

\noindent Dante G. E. Grevino

\noindent \address{Departamento de Matemática, Facultad de Ciencias Exactas y Naturales, Universidad de Buenos Aires,\\
  Buenos Aires, Argentina}

\noindent \email{dantegabriel\_17@hotmail.com}

\bigskip

\noindent Herbert Spohn

\noindent \address{Zentrum Mathematik and Physics Department \\
Technical University Munich\\
M\"{u}nchen, Germany}

\noindent \email{spohn@tum.de}

\end{document}